\documentclass[10pt,a4paper]{article}
\usepackage{amsmath,amsfonts,amsthm,amssymb, mathtools}
\usepackage{mathrsfs, graphicx,color,latexsym, tikz, calc,
}
\usepackage{tikz-cd}
\usepackage{graphicx}
\usepackage{epstopdf}
\usepackage{enumerate}
\usepackage{caption}
\usepackage{hyperref}
\usepackage{float}

\usepackage{amscd} 
\usepackage{titlesec}

\titleformat{\paragraph}
  {\normalfont\bfseries\boldmath} 
  {\theparagraph}        
  {1em}                  
  {}                      
  []                     

\titlespacing*{\paragraph}
  {0pt}                   
  {3.25ex plus 1ex minus .2ex}  
  {1.5ex plus .2ex}

\usetikzlibrary{shadows}
\usetikzlibrary{patterns,arrows,decorations.pathreplacing}
\usepackage{ulem}

\newtheorem{theorem}{Theorem}[section]
\newtheorem{lemma}[theorem]{Lemma}

\newtheorem{remark}{Remark}[section] 
\newtheorem{problem}[theorem]{Problem}

\newtheorem{prop}[theorem]{Proposition}

\allowdisplaybreaks[4]

\date{}
\title{\bfseries
Stabilities of the Kleitman diameter theorem%
\thanks{E-mail addresses:
\url{wuyjmath@163.com} (Y. Wu),
\url{ytli0921@hnu.edu.cn} (Y. Li),
\url{fenglh@163.com} (L. Feng),
\url{jiuqiang68@126.com} (J. Liu),
\url{yuguihai@126.com} (G. Yu).}
}

\author{%
{\normalsize
Yongjiang Wu$^a$,\ 
Yongtao Li$^b$,\ 
Lihua Feng$^a$\thanks{Corresponding author.},\ 
Jiuqiang Liu$^c$,\ 
Guihai Yu$^d$
}\\[3mm]
{\footnotesize
\begin{tabular}{@{}c@{}}
$^a$School of Mathematics and Statistics, HNP-LAMA,
Central South University, Changsha, Hunan 410083, China
\\[0.8mm]
$^b$Yau Mathematical Sciences Center,
Tsinghua University, Beijing 100084, China
\\[0.8mm]
$^c$Department of Mathematics,
Eastern Michigan University, Ypsilanti, MI 48197, USA
\\[0.8mm]
$^d$College of Big Data Statistics,
Guizhou University of Finance and Economics,
Guiyang, Guizhou 550025, China
\end{tabular}
}}

\begin{document}
\maketitle
\begin{abstract}
Let $\mathcal{F}$ be a family of subsets of $[n]$. The diameter of $\mathcal{F}$ is defined as the maximum size of the symmetric difference between any two of  its members. Resolving a conjecture of Erd\H{o}s,  Kleitman  determined the maximum size of a family with fixed diameter. Specifically, if $2\leq s\leq n-2$ and
$\mathcal{F}\subseteq2^{[n]}$ is a family with diameter $s$, then for $s=2d$, $|\mathcal{F}|\le \sum_{i=0}^d {n \choose i}$; 
for $s=2d+1$,
$|\mathcal{F}|\le \sum_{i=0}^d {n \choose i} + {n-1 \choose d}$. 
This result, known as the Kleitman diameter theorem,  generalizes both the Katona union theorem and the Erd\H{o}s--Ko--Rado theorem. In 2017, Frankl provided a complete characterization of the extremal families for Kleitman's theorem and established a stability result. 
In this paper, 
we determine the extremal families for Frankl's theorem and prove a further stability result of Kleitman's theorem. This solves a recent problem proposed by Li and Wu and provides the second stability for the Kleitman diameter theorem. \\
\end{abstract}

 {\bf AMS Classification}:  05C65; 05D05 

{\bf Key words}: Extremal set theory; Kleitman diameter theorem; Stability

\section{Introduction}

Extremal set theory 
is an important topic of combinatorics, and it always deals with the problems on families of subsets of an $n$-element set. One of the central problems is to determine or estimate the maximum size of a family that satisfies certain constraints, including the restricted intersection or union. 
To begin with, we introduce the following notation. 
For two integers $a\leq b$, let  $[a, b]=\{a, a+1, \ldots, b\}$ and denote $[n]=[1,n]$. Let $2^{[n]}$ denote the power set of $[n]$. For any $0 \leq k \leq n$, let $\binom{[n]}{k}$ denote the collection of all $k$-element subsets of $[n]$. A family $\mathcal{F} \subseteq 2^{[n]}$ is called $k$-\textit{uniform} if $\mathcal{F} \subseteq \binom{[n]}{k}$. A family $\mathcal{F} \subseteq 2^{[n]}$ is called $t$-\textit{intersecting} if $|F\cap F^{\prime}|\geq t$
 for all $F, F^{\prime}\in \mathcal{F}$. If $t=1$, $\mathcal{F}$ is simply called \textit{intersecting}. Two families  $\mathcal{F}, \mathcal{G}\subseteq 2^{[n]}$ are  \textit{isomorphic} if there exists a permutation $\sigma$ on $[n]$ such that $\mathcal{G}=\left\{\{\sigma(x):x\in F\}: F\in \mathcal{F}\right\}$, and we write $\mathcal{F}\cong\mathcal{G}$.

\subsection{Restricted intersection: Erd\H{o}s-Ko-Rado's theorem}

In 1961,  Erd\H{o}s, Ko and Rado \cite{E61} determined the maximum  size of an 
intersecting family of $k$-subsets of $[n]$. 
This result is one of the cornerstones of extremal set theory, 
and since then, a vast number of results
of a similar flavour have been investigated for a range of different mathematical structures; see the surveys \cite{Ellis2021,FT16}. 

\begin{theorem}[Erd\H{o}s, Ko and Rado \cite{E61}] \label{E61}
Let $k \geq 2$ and $n \geq 2 k$ be integers. Let $\mathcal{F} \subseteq\binom{[n]}{k}$ be an intersecting family. Then
$$
|\mathcal{F}|\leq\binom{n-1}{k-1}.
$$
For $n \geq 2 k+1$,  equality holds if and only if $\mathcal{F}=\left\{F\in \binom{[n]}{k}: x\in F\right\}$ for some $x \in [n]$.
\end{theorem}

A family $\mathcal{F}$ that satisfies  $\mathcal{F}=\left\{F\in \binom{[n]}{k}: x\in F\right\}$ is called a \textit{full star} centered at $x$, and $\mathcal{F}$ is called \textit{EKR} if
$\mathcal{F}$ is contained in a full star. Such a family is a ``trivial" example of an intersecting family. If $\mathcal{F}$ is not EKR, then  $\cap_{F\in\mathcal{F}}F=\emptyset$ and $\mathcal{F}$ is called \textit{non-trivial}.
In 1967,  Hilton and Milner \cite{H67} determined the maximum size of a non-trivial $k$-uniform intersecting family.

\begin{theorem}[Hilton and Milner \cite{H67}] \label{H67}
Let $k \geq 2$ and $n \geq 2 k+1$ be integers. Let $\mathcal{F} \subseteq\binom{[n]}{k}$ be an intersecting family.
If $\mathcal{F}$ is not EKR, then
$$
|\mathcal{F}|\leq\binom{n-1}{k-1}-\binom{n-k-1}{k-1}+1.
$$ 
\end{theorem}

This bound is substantially smaller than the one 
given by Theorem \ref{E61}. 
Hilton and Milner also characterized  the  extremal families. Equality in Theorem \ref{H67} holds 
if and only if
$\mathcal{F}\cong \mathcal{HM}(n, k)$,
and additionally  when $k=3$, 
$\mathcal{F}$ may also be isomorphic to $\mathcal{T}(n, 3)$,
where $ \mathcal{HM}(n, k) = 
\big\{F \in\binom{[n]}{k}: 1 \in F, F \cap[2, k+1] \neq \emptyset \big\} \cup\{[2, k+1]\}$ and $\mathcal{T}(n, 3) = 
\big\{F \in\binom{[n]}{3}:|F \cap[3]| \geq 2\big\}$. 
A family $\mathcal{F}$ is called \textit{HM} if
$\mathcal{F}$ is isomorphic to a subfamily of $\mathcal{HM}(n, k)$. 
We write $\mathcal{F} \nsubseteq \mathcal{T}(n, 3)$ to indicate that $\mathcal{F}$ is not isomorphic to a subfamily of $\mathcal{T}(n, 3)$.
 In 2017, Han and Kohayakawa \cite{H17} established  a further stability result for Theorem \ref{H67}.

\begin{theorem}[Han and Kohayakawa \cite{H17}] 
\label{H17}
Let $k \geq 3$ and $ n\geq 2 k+1$ be integers. Let $\mathcal{F} \subseteq\binom{[n]}{k}$ be an intersecting family. If $\mathcal{F}$ is neither EKR nor HM, and if $k=3, \mathcal{F} \nsubseteq \mathcal{T}(n, 3)$, then
$$
|\mathcal{F}|\leq\binom{n-1}{k-1}-\binom{n-k-1}{k-1}-\binom{n-k-2}{k-2}+2.
$$
\end{theorem}

Recently, the current authors \cite{W240} 
provided a unified approach to revisiting 
Theorems \ref{H67} and \ref{H17}.  
Apart from the above results, 
many stability results (together with a wide variety of different methods) for the 
Erd\H{o}s-Ko-Rado theorem have been obtained in the literature; see, 
e.g., \cite{EKL2019,FW2024-EUJC, GXZ2024,H24,Kee2008,KL2020,KM17,K19,KZ2018,Mub2007}. Moreover, the stability results for $t$-intersecting families have also been well studied in recent years \cite{BL2023,BL2022,CLW2021,CLLW2022,FW2024+}. 
Besides being interesting in their own right, such stability
results can often serve as key ingredients for determining the exact Tur\'{a}n-type results for set families.  
In this paper, we shall pay attention to the stability results on the extremal problems for non-uniform families with restricted union and symmetric difference.

\subsection{Restricted union: Katona's theorem}

 Let $\mathcal{F} \subseteq 2^{[n]}$ be a family. We say that $\mathcal{F}$ is $s$-\textit{union} if $|F\cup F^{\prime}|\leq s$ for all $F, F^{\prime}\in \mathcal{F}$.
Obviously, $\mathcal{F}$ is $t$-intersecting if and only if the family $\mathcal{F}^c=\{F^c=[n]\backslash F: F\in \mathcal{F}\}$ is $(n-t)$-union. Since $|\mathcal{F}^c|=|\mathcal{F}|$, 
determining the maximum size of a 
$t$-intersecting family 
is equivalent to determining the maximum size of the corresponding $s$-union family. 
The maximum size of an $s$-union family on $[n]$ can be easily obtained when $s=0,1,n-1$ and $n$.  
In 1964, Katona \cite{K64} determined the maximum size of $s$-union families for $2 \leq s\leq n-2$. (The original formulation in \cite{K64} is given in terms of  intersections of sets rather than unions). 

\begin{theorem}[Katona \cite{K64}] \label{K64}
 Let $2 \leq s \leq n-2$ be an integer. Suppose that  $\mathcal{F} \subseteq 2^{[n]}$ is $s$-union.
 \begin{itemize}
\item[\rm (i)]  If $s=2 d$ for some integer $d \geq 1$, then
$
|\mathcal{F}| \leq \sum_{i=0}^d \binom{n}{i}$.

\item[\rm (ii)]  If $s=2 d+1$ for some integer $d \geq 1$, then $|\mathcal{F}| \leq \sum_{i=0}^d \binom{n}{i}+\binom{n-1}{d}$.
 \end{itemize}
\end{theorem}

 The extremal families achieving these bounds were also determined in \cite{K64}. Specifically, 
 equality in (i) holds only for 
$\mathcal{K}(n, 2 d) = \left\{F\subseteq [n]: |F|\leq d\right\}$, and   
equality in (ii) holds only for 
$\mathcal{K}(n, 2 d+1)=\left\{F\subseteq [n]: |F|\leq d\right\} \cup\left\{F \in\binom{[n]}{d+1}: y \in F\right\}$ for some $y\in [n]$.  
The original proof of Theorem \ref{K64} 
primarily relies on a  shadow theorem for intersecting families. 
Alternative proofs of Katona's theorem 
can be found in \cite{AK1999,AK2005,H20,Wang1977}
and \cite[p.~98]{Bol1986}.
For further generalizations, we refer the interested readers to \cite{BL2023,FT2013,FT2016,KL2020} and the references therein. 
In 2017, Frankl \cite{F171}  determined the
sub-optimal $s$-union families for $2 \leq s\leq n-2$.

\begin{theorem}[Frankl  \cite{F171}] 
\label{F171} Let $2 \leq s \leq n-2$ be an integer. Suppose that  $\mathcal{F} \subseteq 2^{[n]}$ is $s$-union. 
 \begin{itemize}
\item[\rm (i)]  If $s=2 d$ for  $d \geq 1$ and $\mathcal{F} \nsubseteq \mathcal{K}(n, 2 d)$, then
$
|\mathcal{F}| \leq \sum_{i=0}^d\binom{n}{i}-\binom{n-d-1}{d}+1.
$

\item[\rm (ii)]   If $s=2 d+1$ for  $d \geq 1$ and $\mathcal{F} \nsubseteq \mathcal{K}(n, 2 d+1)$, then
$
|\mathcal{F}| \leq \sum_{i=0}^d\binom{n}{i}+\binom{n-1}{d}-\binom{n-d-2}{d}+1.
$
\end{itemize}
 \end{theorem}

The extremal families achieving the bounds in  Theorem \ref{F171} are also characterized. 
Equality in (i) holds if and only if there exists $D \in\binom{[n]}{d+1}$ such that
$
\mathcal{F}=\mathcal{H}(n, 2 d)$, 
and additionally  when  $s=4$, $\mathcal{F}$ may also be isomorphic to  $ \mathcal{H}^*(n, 4)$~\cite{L24},  
where 
\begin{align*}
\mathcal{H}(n, 2 d) &= 
\left\{F\subseteq [n]: |F|\leq d-1\right\}\cup\{D\} \cup\left\{F \in\binom{[n]}{d}: F \cap D \neq \emptyset\right\},\\
\mathcal{H}^*(n, 4) &= \left\{F\subseteq [n]: |F|\leq 1\right\} \cup\left\{F \in\binom{[n]}{2}: F \cap[2] \neq \emptyset\right\} \cup \big\{\{1,2, i\}: i \in[3, n] \big\}.
\end{align*}
Equality in (ii) holds if and only if  there exist $D \in\binom{[n]}{d+1}$ and $y \in[n]\backslash D$ such that
$
\mathcal{F}=\mathcal{H}(n, 2 d+1)$, 
and additionally  when  $s=5$, $\mathcal{F}$ may also be isomorphic to  
$ \mathcal{T}(n, 5)$,
where 
\begin{align*}
\mathcal{H}(n, 2 d+1)  &= 
\left\{F\subseteq [n]: |F|\leq d\right\} \cup\{D\} \cup\left\{F \in\binom{[n]}{d+1}: y \in F, F \cap D \neq \emptyset\right\},\\
\mathcal{T}(n, 5)&= \left\{F\subseteq [n]: |F|\leq 2\right\}\cup\left\{F \in\binom{[n]}{3}:|F \cap[3]| \geq 2\right\}.
\end{align*}
Note that there are two possibilities in the cases $s=4$ and $5$. 
We write $\mathcal{F} \nsubseteq \mathcal{H}^*(n, 4)$ (resp. $\mathcal{F} \nsubseteq \mathcal{T}(n, 5)$) to indicate that $\mathcal{F}$ is not isomorphic to a subfamily of $\mathcal{H}^*(n, 4)$ (resp. $\mathcal{T}(n, 5)$).
 In 2024, Li and Wu \cite{L24} further determined the third optimal $s$-union families.

\begin{theorem}[Li and Wu \cite{L24}] 
\label{L24}
Let $4 \leq s \leq n-2$ be an integer. Suppose that  $\mathcal{F} \subseteq 2^{[n]}$ is an $s$-union family such that $\mathcal{F}\nsubseteq\mathcal{K}(n, s)$ and $\mathcal{F}\nsubseteq\mathcal{H}(n, s)$.
Then the following holds. 
 \begin{itemize}
\item[\rm (i)]  If $s=2 d$ for $d \geq 2$ and further $\mathcal{F}\nsubseteq \mathcal{H}^*(n, 4)$, then
$
|\mathcal{F}| \leq \sum_{i=0}^d\binom{n}{i}-\binom{n-d-1}{d}-\binom{n-d-2}{d-1}+2.
$
\item[\rm (ii)]  If $s=2 d+1$ for  $d \geq 2$  and further $\mathcal{F}\nsubseteq\mathcal{T}(n, 5)$, then
$
|\mathcal{F}| \leq \sum_{i=0}^d\binom{n}{i}+\binom{n-1}{d}-\binom{n-d-2}{d}-\binom{n-d-3}{d-1}+2.
$
 \end{itemize}
\end{theorem}

For $s=2$, we shall mention that a $2$-union family
must be a subfamily of $\mathcal{K}(n,2)$ or
$\mathcal{H}(n,2)$.
For $s=3$, if a $3$-union family $\mathcal{F}$
is neither a subfamily of $\mathcal{K}(n,3)$ nor
of $\mathcal{H}(n,3)$, then $|\mathcal{F}|\leq8$,
with equality if and only if
$\mathcal{F}=2^{\{a,b,c\}}$
for some distinct $a,b,c\in[n]$.
Moreover, the extremal families attaining the bounds
in Theorem~\ref{L24} were also characterized in
\cite{L24}. Here, we do not describe them for simplicity.

\subsection{Restricted diameter: Kleitman's theorem}

For convenience, 
our notation follows that of Frankl \cite{F172}.
Let $A, B$ be two subsets of $[n]$. The \textit{symmetric difference} of $A$ and $B$ is defined as
$$A+B:=(A\backslash B)\cup (B\backslash A).$$
The family $ 2^{[n]}$ is considered as a metric space with the distance  $d(A,B)=|A+B|$ for every $A, B\subseteq [n]$. 
The \textit{diameter} of a family $\mathcal{F} \subseteq 2^{[n]}$, denoted by $\Delta(\mathcal{F})$, is defined as ${\rm max}\{d(A,B): A, B\in \mathcal{F}\}$.  
For any $\mathcal{F} \subseteq 2^{[n]}$ and $S \subseteq [n]$,  the \textit{translation} of $\mathcal{F}$ by $S$ is defined as 
$
\mathcal{F}+S :=\{F+S: F \in \mathcal{F}\}.
$ 
Observe that $|\mathcal{F}+S|=|\mathcal{F}|$ and $\Delta(\mathcal{F}+S)=\Delta(\mathcal{F})$.
In 1966, Kleitman \cite{K66} proved the following theorem,  originally conjectured by Erd\H{o}s.

\begin{theorem}[Kleitman \cite{K66}] \label{K66}  Suppose that $2 \leq s \leq n-2$ and $\mathcal{F} \subseteq 2^{[n]}$ satisfies $\Delta(\mathcal{F})\leq s$.
\begin{itemize}
\item[\rm (i)] If $s=2 d$ for  $d \geq 1$, then
$
|\mathcal{F}| \leq \sum_{i=0}^d \binom{n}{i}.
$

\item[\rm (ii)]  If $s=2 d+1$ for  $d \geq 1$, then
$
|\mathcal{F}| \leq \sum_{i=0}^d\binom{n}{i}+\binom{n-1}{d}.
$
\end{itemize}
\end{theorem}

Since
$|A+B|\le |A\cup B|$, the $s$-union assumption on $\mathcal{F}$ implies  $\Delta(\mathcal{F})\leq s$.  
Thus, Theorem \ref{K66} 
generalizes Theorem \ref{K64}. 
 Kleitman's combinatorial proof transfers Theorem \ref{K66} into  Theorem \ref{K64} by using the so-called down-shift operation (see Section \ref{se31}). 
Recently, Huang, Klurman and Pohoata \cite{H20} gave an  algebraic proof using the Cvetkovi\'{c} spectral bound on the independence number, and also showed several extensions of Theorem \ref{K66} to other allowed distance sets consisting of consecutive integers. 
Theorem \ref{K66}  can be viewed as an isodiametric inequality for discrete hypercubes.  Many generalizations have been considered, such as the $n$-dimensional grid $[m]^n$ with Hamming distance \cite{A98, F80}, as well as
$[m]^n$ and the $n$-dimensional torus $\mathbb{Z}_m^n$ with Manhattan
distance \cite{A92, B93, D90}.  
In 2017, Frankl \cite{F172} proved the following stability result, which is a diameter extension of Theorem \ref{F171}.

\begin{theorem}[Frankl \cite{F172}] \label{F172}
 Let $2 \leq s \leq n-2$ be an integer. Suppose that  $\mathcal{F} \subseteq 2^{[n]}$ satisfies $\Delta(\mathcal{F})\leq s$. Then the only family achieving the equality of Theorem \ref{K66} is the translation of
Katona's family $\mathcal{K}(n, s)$. Moreover, the following stability result holds.
\begin{itemize}
\item[\rm (i)]  If $s=2 d$ and $\mathcal{F}$ is not contained in any translation of $\mathcal{K}(n, 2d)$, then
$
|\mathcal{F}| \leq \sum_{i=0}^d\binom{n}{i}-\binom{n-d-1}{d}+1.
$

\item[\rm (ii)]  If $s=2 d+1$ and $\mathcal{F}$ is not contained in any translation of $\mathcal{K}(n, 2d+1)$, then
$
|\mathcal{F}| \leq \sum_{i=0}^d\binom{n}{i}+\binom{n-1}{d}-\binom{n-d-2}{d}+1.
$
\end{itemize}
\end{theorem}

Theorem \ref{F172} provides the first stability of Kleitman's  theorem. Frankl's proof  is combinatorial and builds on his earlier result  Theorem \ref{F171}. More recently, Gao, Liu and Xu \cite{G23} gave a stability result for Kleitman's  theorem in  the case $s=2 d+1$ using linear algebra methods. 
Inspired by these stability results, Li and Wu \cite{L24} proposed a problem to characterize the second stability of Kleitman's theorem after showing Theorem \ref{L24}. 

\begin{problem}[Li and Wu \cite{L24}] \label{L24P}
Is there a diameter correspondence of Theorem \ref{L24}? 
In other words, is there a further  stability result for Theorem \ref{F172}? 
\end{problem}

The relations of aforementioned theorems can be described as below:  

\begin{table}[H]
\centering
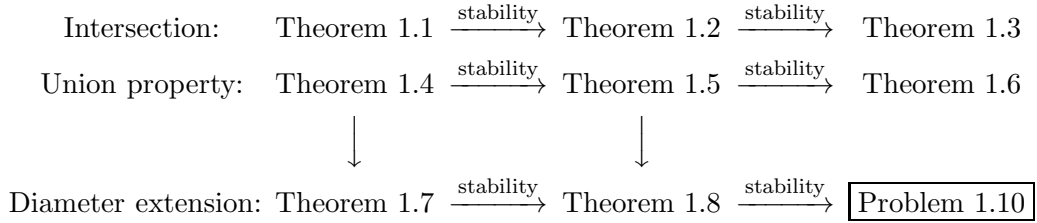
\caption{The framework of the stability results.}
\label{tab-relation}
\begin{tabular}{|c|c|c|c|}
\hline
\textbf{Property} & \textbf{Base Theorem} & \textbf{First stability} & \textbf{Second stability} \\ \hline
Intersection & Theorem \ref{E61} & Theorem \ref{H67} & Theorem \ref{H17} \\ \hline
Union & Theorem \ref{K64} & Theorem \ref{F171} & Theorem \ref{L24} \\ \hline
Diameter & Theorem \ref{K66} & Theorem \ref{F172} & Problem \ref{L24P} \\ \hline
\end{tabular}
\end{table}

In this paper, we resolve this problem  by establishing the second stability  for families with fixed diameter. Our main contribution is twofold: we characterize the extremal families that achieve the bound in Frankl's stability theorem, and we provide sharp upper bounds for families that are not contained in the first-level extremal structures. This result constitutes the second stability for the Kleitman diameter theorem and reveals additional structural complexity, particularly in the odd diameter case, which exhibits a phenomenon not present in the corresponding union version.

\subsection{Main results} 

This paper provides an affirmative  answer to Problem \ref{L24P}. 
A crucial step toward this solution is the characterization of the extremal families in Theorem \ref{F172}, which lays the groundwork for a more refined stability analysis. 
To present our results clearly, we first introduce some necessary notation.
Let $d\geq 2$ be an integer, $ R\in\binom{[n]}{d+2}$ and $y\in R$. Define
\begin{align*}
\mathcal{R}(n, 2 d)=&\left\{F\subseteq [n]: |F|\leq d-2\right\}\cup\{R\} \cup\left\{F\subseteq [n] : y\in F, |F|=d-1 \text{ or } d\right\}\\
&\cup\left\{F\subseteq [n] : y\notin F, F \cap R \neq \emptyset, |F|=d-1 \text{ or } d\right\}.
\end{align*}
For the case $s=4$, we need two additional families.
For any $y\in [3,n]$,  define
\begin{align*}
\mathcal{R}^*(n, 4) =& \{F\subseteq [n]: |F|\leq 1 \} \cup\{\{1,2\}, \{1,y\}, \{2,y\}\}\cup\{F \subseteq [n]: |F|=3, y\in F, F \cap[2] \neq \emptyset\}\\
&\cup \{\{1,2,i\}: i \in[3, n] \},\\
\mathcal{U}^*(n, 4) =&\{ \emptyset,\{1\}, \{2\}, \{y\} \} \cup\{ \{y,i\},i \in [n]\backslash \{y\}\}\cup\{F \subseteq [n]: |F|=2, y\notin F, F \cap[2] \neq \emptyset\}\\
&\cup \{\{1,2,y\} \}\cup\{\{1,2,y, i\}: i \in [n]\backslash \{1,2,y\} \}.
\end{align*}
Recall that $\mathcal{K}(n, 2)=\left\{F\subseteq [n]: |F|\leq 1\right\}$ and $ \mathcal{H}(n, 2)=\left\{\emptyset, \{a\}, \{b\}, \{a, b\}\right\}$
for $a\neq b\in [n]$. Let $c\in [n]\backslash \{a,b\}$ and define $\mathcal{V}(n, 2)=\left\{\{a\}, \{b\}, \{c\}, \{a, b, c\}\right\}.$

The structure of a family $\mathcal{F}\subseteq 2^{[n]}$ with diameter $\Delta(\mathcal{F}) \leq s$ varies significantly with $s$. The cases $s = 0,1,n-1,n$ are straightforward and can be handled directly.
For the nontrivial range $2 \leq s \leq n-2$, the problem becomes substantially more interesting. We first consider the smallest case $s = 2$, which exhibits a simple but instructive structure.

\begin{theorem}\label{ma2}
Suppose that  $\mathcal{F} \subseteq 2^{[n]}$ satisfies $\Delta(\mathcal{F})\leq 2$. Then $\mathcal{F}$ must be contained in a translation of $\mathcal{K}(n, 2)$, $\mathcal{H}(n, 2)$ or $\mathcal{V}(n, 2)$.
\end{theorem}

For $s \geq 3$, the situation becomes significantly more complex. The following result characterizes the extremal families of Theorem \ref{F172}.  

\begin{theorem}\label{ma3}
 Suppose that $3 \leq s \leq n-2$ and $\mathcal{F} \subseteq 2^{[n]}$ satisfies $\Delta(\mathcal{F})\leq s$. Furthermore, $\mathcal{F}$ is not contained in any translation of $\mathcal{K}(n, s)$.
Then the following three statements hold.

 \begin{itemize}
\item[\rm (i)] For $s=3$, the only family achieving equality of Theorem \ref{F172} is the translation of $\mathcal{H}(n, 3)$. 

\item[\rm (ii)] For $s=4$, the families achieving equality of Theorem \ref{F172} are the translations of $\mathcal{H}(n, 4)$, $\mathcal{R}(n, 4)$, $\mathcal{H}^*(n, 4)$, $\mathcal{R}^*(n, 4)$, $\mathcal{U}^*(n, 4)$.
For $s=2d$ ($d\geq 3$), the families achieving equality of Theorem \ref{F172} are the translations of $\mathcal{H}(n, 2d)$ and $\mathcal{R}(n, 2d)$. 

\item[\rm (iii)] For $s=5$, the families achieving equality of Theorem \ref{F172} are the translations of $\mathcal{H}(n, 5)$ and $\mathcal{T}(n, 5)$. For $s=2d+1$ ($d\geq 3$), the only family achieving equality in Theorem \ref{F172} is the translation of $\mathcal{H}(n, 2d+1)$. 
\end{itemize}
\end{theorem}

Note that
$\mathcal{U}^*(n,4)=\mathcal{R}^*(n,4)+\{y\}$.
Hence, these two families belong to the same translation class.
We retain both representatives because they arise
separately in the down-shift argument.
Building on Theorem \ref{ma3}, we now present our main result: a second-level stability theorem that provides sharp upper bounds for families that are not contained in any of these extremal families.

\begin{theorem}[Main result] \label{ma1}
 Suppose that $3 \leq s \leq n-2$ and $\mathcal{F} \subseteq 2^{[n]}$ satisfies $\Delta(\mathcal{F})\leq s$ and $\mathcal{F}$ is not contained in any translation of $\mathcal{K}(n, s)$. 
Then the following stability result holds. 

 \begin{itemize}
\item[\rm (i)] For $s=3$, 
    if  $\mathcal{F}$ is not contained in any translation of $\mathcal{H}(n, 3)$, then  $|\mathcal{F}|\leq 8$. 

\item[\rm (ii)] For $s=2d$,
    if  $\mathcal{F}$ is not contained in any translation of $\mathcal{H}(n, 4)$, $\mathcal{R}(n, 4)$, $\mathcal{H}^*(n, 4)$, $\mathcal{R}^*(n, 4)$, $\mathcal{U}^*(n, 4)$ for $d=2$, and
$\mathcal{F}$ is not contained in any translation of $\mathcal{H}(n, 2d)$ and $\mathcal{R}(n, 2d)$ for $d\geq 3$, then
$$
|\mathcal{F}| \leq \sum_{i=0}^d\binom{n}{i}-\binom{n-d-1}{d}-\binom{n-d-2}{d-1}+2.
$$

 \item[\rm (iii)] For $s=2d+1$,
if  $\mathcal{F}$ is not contained in any translation of $\mathcal{H}(n, 5)$ and $\mathcal{T}(n, 5)$ for $d=2$, and
$\mathcal{F}$ is not contained in any translation of $\mathcal{H}(n, 2d+1)$ for $d\geq 3$, then
\begin{align*}
|\mathcal{F}|\leq {\rm max} &\left\{\sum_{i=0}^d\binom{n}{i}+\binom{n-1}{d}-\binom{n-d-2}{d}-\binom{n-d-3}{d-1}+2,\right.\\
&\left.\sum_{i=0}^d\binom{n}{i}+\binom{n-1}{d}-\binom{n-d-2}{d}-\binom{n-d-3}{d}+1\right\}.
\end{align*}
\end{itemize}
\end{theorem}

Since Theorem  \ref{ma1} implies Theorem \ref{ma3}, it suffices to prove Theorem  \ref{ma1}.
To help the reader navigate the classification results in Theorems  \ref{ma3} and  \ref{ma1}, we summarize the extremal and near-extremal families for different parameter ranges in Table \ref{tab:summary}. Here, $\mathcal{K}(n,s)$ denotes the extremal family of Theorem \ref{K64}, $\mathcal{H}(n,s)$, $\mathcal{H}^*(n,4)$, $\mathcal{T}(n,5)$ denote the extremal families of Theorem \ref{F171}, and $\mathcal{R}(n,2d)$, $\mathcal{R}^*(n,4)$, $\mathcal{U}^*(n,4)$, $\mathcal{V}(n, 2)$ are the exceptional families defined above. 

\begin{table}[htbp]
\centering
 \caption{Summary of extremal families for Kleitman's theorem (all up to translation).}
\label{tab:summary}
\begin{tabular}{|c|c|c|}
\hline
\textbf{Parameter range} & \textbf{Extremal families } & \textbf{Near-extremal families} \\
\hline
$s = 2$ & $\mathcal{K}(n,2)$ & $\mathcal{H}(n,2),\ \mathcal{V}(n,2)$  \\
\hline
$s = 3$ & $\mathcal{K}(n,3)$ & $ \mathcal{H}(n,3)$\\
\hline
$s = 4$ & $\mathcal{K}(n,4)$ & $\mathcal{H}(n,4),\  \mathcal{H}^*(n,4),\ \mathcal{R}(n,4),\ \mathcal{R}^*(n,4),\ \mathcal{U}^*(n,4)$ \\
\hline
$s = 5$ & $\mathcal{K}(n,5)$ & $\mathcal{H}(n,5),\ \mathcal{T}(n,5)$  \\
\hline
$s = 2d\ (d \geq 3)$ & $\mathcal{K}(n,2d)$ & $\mathcal{H}(n,2d),\ \mathcal{R}(n,2d)$  \\
\hline
$s = 2d+1\ (d \geq 3)$ & $\mathcal{K}(n,2d+1)$ & $\mathcal{H}(n,2d+1)$ \\
\hline
\end{tabular}
\end{table}

 Upon comparing the diameter extensions with their union counterparts, it is evident that the upper bound established in Theorem \ref{K66} precisely coincides with that of Theorem \ref{K64}. Furthermore, a parallel correspondence emerges between Theorems \ref{F172} and \ref{F171}, demonstrating an identical bounding behavior. However, a nuanced distinction arises when examining Theorems \ref{ma1} and \ref{L24}, where the former exhibits an additional upper bound for the odd case, thereby introducing a more refined classification.
This distinction stems from 
a fundamental difference in the structural symmetry of the target families obtained after down-shifts. Recall that by repeatedly applying down-shifts (to be detailed in Section \ref{se31}), any family $\mathcal{F}$ with $\Delta(\mathcal{F}) \leq s$ yields a family $\mathcal{G}$ satisfying $|\mathcal{G}| = |\mathcal{F}|$ and $\Delta(\mathcal{G}) \leq s$. Moreover, $\mathcal{G}$ can be chosen such that either it is a complex, or there exists an element $j \in [n]$ for which $S_j(\mathcal{G})$ is contained in a specific  family.

In the even case, all target  families, including $\mathcal{K}(n,2d), \mathcal{H}(n,2d), \mathcal{R}(n,2d)$, as well as  $\mathcal{H}^*(n,4),  \mathcal{R}^*(n,4), \mathcal{U}^*(n,4)$ for $s=4$, possess a high degree of symmetry across all layers. For any given size, the sets in that layer either all contain a fixed element, or all intersect a fixed set $D$, or are symmetric with respect to a small set of points. This uniformity means that when $S_j(\mathcal{G})$ is contained  in such a symmetric family, all possible deviations behave similarly in the counting argument, yielding a single bound.

In the odd case, although  the families $\mathcal{K}(n,2d+1)$ and $\mathcal{T}(n,5)$ are symmetric,  the family $\mathcal{H}(n,2d+1)$ lacks this layer-wise symmetry. At size
$d+1$, it contains two types of sets: those containing the distinguished element $y$, and
exactly one set $D$ that does not contain $y$. This asymmetry leads to the following  phenomenon: a
$(d+2)$-set mapping down under $S_j$ can become either $D$  or a $y$-containing set, and these two possibilities affect the counting argument differently, resulting in two distinct bounds.
The family $\mathcal{Q}(n,2d+1)$ in Remark \ref{remad1} exemplifies this phenomenon and  attains the additional upper bound in Theorem \ref{ma1} (iii).

\section{Preliminaries}\label{se31}
Throughout this paper, we fix the following  notation.
Given a family   $\mathcal{F} \subseteq 2^{[n]}$  and $1\leq i\neq j\leq n$, define 
\begin{align*}
&\mathcal{F}(\bar{i})=\left\{F: i \notin F \in \mathcal{F}\right\},~\mathcal{F}(i)=\left\{F\backslash \{i\}: i \in F \in \mathcal{F}\right\},\\
&\mathcal{F}(i,j)=\left\{F\backslash \{i,j\}: i, j\in F \in \mathcal{F}\right\},~\mathcal{F}(\bar{i},\bar{j})=\left\{F: i, j\notin F \in \mathcal{F}\right\},\\
&\mathcal{F}(i,\bar{j})=\mathcal{F}(\bar{j},i)=\left\{F\backslash \{i\}: i \in F, j\notin F, F \in \mathcal{F}\right\}.
\end{align*}
We also write $\mathcal{F}_k=\mathcal{F}\cap \binom{[n]}{k}$ for the collection of $k$-sets in $\mathcal{F}$.
With this notation in place, we first examine the simplest nontrivial case: families with diameter at most two. This serves both as a proof of Theorem \ref{ma2} and as a warm-up for the more complex arguments to follow.

\subsection{Families with  diameter at most two}
\begin{proof}[\bf Proof of Theorem \ref{ma2}]
Recall that
$\mathcal{K}(n, 2)=\left\{F\subseteq [n]: |F|\leq 1\right\}$, 
$\mathcal{H}(n, 2)=\left\{\emptyset, \{a\}, \{b\}, \{a, b\}\right\}$ and $ \mathcal{V}(n, 2)=\left\{\{a\}, \{b\}, \{c\}, \{a, b, c\}\right\}$. 
Suppose that  $\mathcal{F} \subseteq 2^{[n]}$ satisfies $\Delta(\mathcal{F})\leq 2$ with $n\geq 4$. By Theorem \ref{F172}, if $|\mathcal{F}|\geq 5$, then $\mathcal{F}$ must be contained in a translation of $\mathcal{K}(n, 2)$. Thus, we may assume that $|\mathcal{F}|\leq 4$. Since $A+B=(A+C)+(B+C)$, we may translate $\mathcal{F}$ such that $\emptyset\in \mathcal{F}$. Then ${\rm max}\{|F|: F\in \mathcal{F}\}\leq 2$ and $|\mathcal{F}_2|\leq 3$. If $|\mathcal{F}_2|\leq 1$, then
$\mathcal{F} \subseteq \mathcal{K}(n, 2)$ or $\mathcal{F} \subseteq \mathcal{H}(n, 2)$. If $|\mathcal{F}_2|= 2$, write $\mathcal{F}_2=\{\{a,b\}, \{a,c\}\} $. Then $\mathcal{F}_1\subseteq\{\{a\}\}$ since $\Delta(\mathcal{F})\leq 2$, giving
$$\mathcal{F}\subseteq \left\{\emptyset, \{a\},\{a, b\}, \{a, c\}\right\}=\left\{\emptyset, \{a\},\{b\}, \{c\}\right\}+\{a\}\subseteq\mathcal{K}(n, 2)+\{a\}.$$
If $|\mathcal{F}_2|= 3$, then either $\mathcal{F}_2=\{\{a,b\}, \{a,c\}, \{a,d\}\} $ or $\mathcal{F}_2=\{\{a,b\}, \{a,c\}, \{b,c\}\} $. It follows that either
$$\mathcal{F}= \left\{\emptyset, \{a,b\},\{a, c\}, \{a, d\}\right\}=\left\{\{a\},\{b\}, \{c\}, \{d\}\right\}+\{a\}\subseteq\mathcal{K}(n, 2)+\{a\},$$
or
$\mathcal{F}= \left\{\emptyset, \{a,b\},\{a, c\}, \{b, c\}\right\}=\left\{\{a\},\{b\}, \{c\}, \{a,b,c\}\right\}+\{a, b, c\}=\mathcal{V}(n, 2)+\{a, b, c\}.$ This completes the proof.
\end{proof}

The above analysis for $s=2$ illustrates how translation and structural constraints interact. For larger diameters, a more systematic tool is needed to relate diameter conditions to union conditions. This is provided by the down-shift operator, which we now introduce.

\subsection{Down-shift operator and a useful bound}\label{subse22}

Let  $\mathcal{F} \subseteq 2^{[n]}$ be a family and $j\in [n]$. The \textit{down-shift operation} $S_{j}$, also known as the squashing operation, introduced by Kleitman \cite{K66}, is defined by 
$S_{j}(\mathcal{F})=\left\{S_{j}(F): F \in \mathcal{F}\right\},$
where
$$
S_{j}(F)= \begin{cases}
F \backslash\{j\}
& \text {if } j \in F \text { and } F \backslash\{j\}\notin \mathcal{F}, \\
F & \text {otherwise. }\end{cases}
$$
A useful observation is that applying down-shifts to certain  families yields other known constructions. For instance, $S_y(\mathcal{R}(n, 2d))=\mathcal{H}(n, 2d)$, $S_y(\mathcal{R}^*(n, 4))=\mathcal{H}^*(n, 4)$, $S_y(\mathcal{U}^*(n, 4))=\mathcal{H}^*(n, 4)$ and $S_c(\mathcal{V}(n, 2))=\mathcal{H}(n, 2)$.
A family $\mathcal{F}\subseteq 2^{[n]}$ is called a \textit{complex} if $E\subseteq F\in \mathcal{F}$ implies that  $E\in \mathcal{F}$. 
By repeatedly applying down-shifts, we may transform any family into a complex without increasing its diameter or changing its size. This crucial property is stated in the following lemma.

\begin{lemma}[ \cite{K66, F172}] \label{K661}
Let  $\mathcal{F} \subseteq 2^{[n]}$ be a family and $j\in [n]$. Then
$|S_{j}(\mathcal{F})|=|\mathcal{F}|$ and $\Delta\left(S_{j}(\mathcal{F})\right)\leq \Delta(\mathcal{F})$. Moreover, if $\mathcal{F}$ is a complex, then $|E\cup F|\leq|\Delta(\mathcal{F})|$  for all $E, F\in \mathcal{F}$.
\end{lemma}

The significance of this lemma is that it allows us to replace a diameter constraint with a union constraint after sufficient down-shifts, thereby connecting our problem to Katona's theorem and its stability versions.
In the proof of Theorem \ref{ma2}, we saw that translating the family so that $\emptyset$ is included simplified the analysis considerably. For Theorem \ref{ma1}, a similar normalization is helpful. 
Let $3 \leq s \leq n-2$ be an integer. Suppose that  $\mathcal{F} \subseteq 2^{[n]}$ satisfies $\Delta(\mathcal{F})\leq s$. Moreover, $\mathcal{F}$ is not contained in any translation of $ \mathcal{K}(n, s)$, $\mathcal{H}(n, s)$, or any  of the exceptional families depending on $s$: $\mathcal{R}(n, 2d)$ for $s=2d$; $\mathcal{T}(n, 5)$ for $s=5$; and $\mathcal{H}^*(n, 4)$, $\mathcal{R}^*(n, 4)$, $\mathcal{U}^*(n, 4)$
 for $s=4$. Let us define
$$
S=\{i:|\mathcal{F}(i)|>|\mathcal{F}(\bar{i})|\} \subseteq[n].
$$
Since  $|\mathcal{F}+S|=|\mathcal{F}|$, $\Delta(\mathcal{F}+S)=\Delta(\mathcal{F})$ and $|(\mathcal{F}+S)(i)| \leq|(\mathcal{F}+S)(\bar{i})|$ for all $i \in[n]$,  
we may
assume that $|\mathcal{F}(i)| \leq $ $|\mathcal{F}(\bar{i})|$ for all $i \in[n]$.

By Lemma \ref{K661}, we know that the down-shift operation maintains the size of a family and does not
increase its diameter.
If $\mathcal{F}$ is not a complex, then applying down-shifts $S_j,  j \in [n]$ repeatedly to $\mathcal{F}$, either we end up with a complex $\mathcal{C}$ of the same size and $\mathcal{C}\nsubseteq \mathcal{K}(n, s)$, $\mathcal{C}\nsubseteq\mathcal{H}(n, s)$, $\mathcal{C}\nsubseteq\mathcal{R}(n, 2d)$ if $s=2d$, in addition, $\mathcal{C}\nsubseteq\mathcal{T}(n, 5)$ if $s=5$, 
 $\mathcal{C}\nsubseteq\mathcal{H}^*(n, 4)$, $\mathcal{C}\nsubseteq\mathcal{R}^*(n, 4)$, $\mathcal{C}\nsubseteq\mathcal{U}^*(n, 4)$ if $s=4$;
or at some point obtain a family $\mathcal{E}$ of the same size satisfying $\mathcal{E}\subseteq \mathcal{K}(n, s)$, or $\mathcal{E}\subseteq\mathcal{H}(n, s)$,  or $\mathcal{E}\subseteq\mathcal{T}(n, 5)$ if $s=5$, or $\mathcal{E}\subseteq\mathcal{R}(n, 2d)$ if $s=2d$, or $\mathcal{E}\subseteq\mathcal{H}^*(n, 4)$, or $\mathcal{E}\subseteq\mathcal{R}^*(n, 4)$,  or $\mathcal{E}\subseteq\mathcal{U}^*(n, 4)$ if $s=4$. In these latter cases, we backtrack and end up with the first such $\mathcal{E}$. Let $\mathcal{E}=S_j(\mathcal{G})$. 
Recall that $|\mathcal{F}(i)| \leq $ $|\mathcal{F}(\bar{i})|$ for all $i \in[n]$. This is not altered by the down-shifts.
Thus, we have $|\mathcal{G}(i)| \leq $ $|\mathcal{G}(\bar{i})|$ for all $i \in[n]$.
In the proof of Theorem~\ref{ma1}, we will frequently encounter situations where $|\mathcal{G}(\bar{j})|$ is small for some $j\in [n]$. In such cases, a simple counting argument gives $|\mathcal{F}| = |\mathcal{G}| \leq 2|\mathcal{G}(\bar{j})|$, which often already meets the required bound. To avoid repeating this calculation, we include it in the following proposition.

\begin{prop}\label{prop1}
Let $3 \leq s \leq n-2$ be an integer. Let $\mathcal{F} \subseteq 2^{[n]}$ be a family with $\Delta(\mathcal{F}) \leq s$, and let $\mathcal{G}$ be obtained from $\mathcal{F}$ by repeated down-shifts such that $|\mathcal{G}(i)| \leq |\mathcal{G}(\bar{i})|$ for all $i \in [n]$. Then the following hold.

\begin{enumerate}[\rm (i)]
\item If $s = 2d$ and there exists $j \in [n]$ such that $|\mathcal{G}(\bar{j})| \leq \sum_{i=0}^{d-1} \binom{n-1}{i} + 2$, then
$$
|\mathcal{F}| \leq \sum_{i=0}^{d-1}\binom{n}{i}+\binom{n-1}{d-1}+4.
$$

\item If $s = 2d+1$ and there exists $j \in [n]$ such that $|\mathcal{G}(\bar{j})| \leq \sum_{i=0}^{d-1} \binom{n-1}{i} +  \binom{n-1}{d}- \binom{n-d-2}{d}$, then
$$
|\mathcal{F}| \leq \sum_{i=0}^{d} \binom{n}{i} +\binom{n-1}{d}-2\binom{n-d-2}{d}.
$$
\end{enumerate}
\end{prop}
\begin{proof}
Since $|\mathcal{G}| \leq 2|\mathcal{G}(\bar{j})|$ by the assumption $|\mathcal{G}(i)| \leq |\mathcal{G}(\bar{i})|$ for all $i$, substituting the given bounds on $|\mathcal{G}(\bar{j})|$ and applying the binomial identity $2\sum_{i=0}^{d-1}\binom{n-1}{i} = \sum_{i=0}^{d-1}\binom{n}{i} + \binom{n-1}{d-1}$ yields the desired inequalities.
\end{proof}

Note that the upper bounds in Proposition \ref{prop1} are strictly smaller than the corresponding bounds in Theorem \ref{ma1}. Consequently, if the conditions of Proposition \ref{prop1} are satisfied, then the desired inequalities in Theorem \ref{ma1} already hold and there is nothing further to prove. 

\subsection{Cross-intersecting families}

When we analyze the layered structure of families obtained after down-shifts, we frequently encounter pairs of subfamilies that lie in different levels and have nonempty intersection constraints. This motivates the study of cross-intersecting families.
Two families $\mathcal{F}, \mathcal{G} \subseteq 2^{[n]}$ are said to be  \textit{cross-$t$-intersecting} if $|F\cap G|\geq t$
 for all $F\in \mathcal{F}$ and $G\in \mathcal{G}$. If $t=1$, then $\mathcal{F}$ and $\mathcal{G}$ are simply called $\textit{cross-intersecting}$.  In our proofs, such pairs will typically consist of sets of different sizes, and we will need sharp bounds on their sizes.
The following lemma, due to  Frankl \cite{F16}, provides a useful inequality for cross-intersecting families when one of them has a stronger intersection property.

\begin{lemma}[\cite{F16}] \label{F16}
Let $t$ be a non-negative integer. Let $k \geq 1$ and $n \geq 2 k+t$ be positive integers.  Let $\mathcal{F} \subseteq\binom{[n]}{k+t}$ and $\mathcal{G} \subseteq\binom{[n]}{k}$ be  cross-intersecting families. If $\mathcal{F}$ is $(t+1)$-intersecting and $|\mathcal{F}| \geq 1$, then
$$
|\mathcal{F}|+|\mathcal{G}| \leq\binom{n}{k}-\binom{n-k-t}{k}+1.
$$
\end{lemma}

A classical special case is when both families are uniform of the same size. This is known  as the Hilton–Milner inequality.

\begin{lemma}[\cite{H67}] \label{H672}
Let $k \geq 1$ and $n \geq 2 k$ be positive integers.  Let $\mathcal{F} \subseteq\binom{[n]}{k}$ and $\mathcal{G} \subseteq\binom{[n]}{k}$ be  non-empty cross-intersecting families. Then
$$
|\mathcal{F}|+|\mathcal{G}| \leq\binom{n}{k}-\binom{n-k}{k}+1.
$$
Moreover, if $n\ge 2k+1$ and  $\mathcal{F}\cap \mathcal{G}=\emptyset $, then the above inequality holds strictly. 
\end{lemma}

 In our proof, we will often encounter situations where both families have size at least two. The following sharpenings of Lemma \ref{H672}, established by the present authors \cite{W240}, will be essential.

\begin{lemma}[\cite{W240}] \label{W231}
Let $k \geq 3$ and $n \geq 2 k$ be positive integers.  Let $\mathcal{F} \subseteq\binom{[n]}{k}$ and $\mathcal{G} \subseteq\binom{[n]}{k}$ be  cross-intersecting families. Suppose that $|\mathcal{F}| \geq 2$ and $|\mathcal{G}| \geq 2$. Then
$$
|\mathcal{F}|+|\mathcal{G}| \leq\binom{n}{k}-\binom{n-k}{k}-\binom{n-k-1}{k-1}+2.
$$
Moreover, if $n\ge 2k+1$ and $|\mathcal{F}\cap \mathcal{G}|\le 1$, then the above inequality holds strictly. 
\end{lemma}

\begin{lemma}[\cite{W240}] \label{HK32}
Let $n \geq 5$ be an integer. Let $\mathcal{F} \subseteq\binom{[n]}{2}$ and $\mathcal{G} \subseteq\binom{[n]}{2}$ be  cross-intersecting families and $|\mathcal{F} \cap \mathcal{G}|\leq 2$. Suppose that $|\mathcal{F}| \geq 2$ and $|\mathcal{G}| \geq 2$. Then
$$
|\mathcal{F}|+|\mathcal{G}| \leq\binom{n}{2}-\binom{n-2}{2}-\binom{n-3}{1}+2.
$$
Moreover, if $|\mathcal{F}\cap \mathcal{G}|\le 1$, 
then the above inequality holds strictly. 
\end{lemma}

To handle cross-intersecting families of different sizes, a powerful tool is the lexicographic order.  For any $F, G\in \binom{[n]}{k}$, we say that $F$ is smaller than $G$ in the \textit{lexicographic order} if $\min (F \backslash G)<\min (G \backslash F)$ holds, 
where $\min(X)$ is the minimum of elements of $X$.  For $0\leq m\leq \binom{n}{k}$, let $\mathcal{L}(n, k, m)$ be the family of the first $m$ $k$-sets in the lexicographic order. 

\begin{lemma}[\cite{FK2017, H76}] \label{Ln}
Let $k, \ell, n$ be positive integers with $n>k+\ell$. If $\mathcal{F} \subseteq\binom{[n]}{k}$ and $\mathcal{G} \subseteq\binom{[n]}{\ell}$ are cross-intersecting, then $\mathcal{L}(n, k,|\mathcal{F}|)$ and $\mathcal{L}(n, \ell,|\mathcal{G}|)$ are cross-intersecting.
\end{lemma}

With this tool in hand, we can now establish a bound needed for the special case $s=5$.

\begin{lemma}\label{Last}
Suppose that $n \geq 6$. Let $\mathcal{F} \subseteq\binom{[n]}{3}$ and $\mathcal{G} \subseteq\binom{[n]}{2}$ be cross-intersecting families. 
\begin{itemize}
\item[\rm(i)]  If $1 \leq |\mathcal{F}| \leq 2n-5$, then
$
|\mathcal{F}|+|\mathcal{G}| \leq\binom{n}{2}-\binom{n-3}{2}+1.
$
\item[\rm(ii)]  If $2n-4 \leq |\mathcal{F}| \leq 3n-9$, then
$
|\mathcal{F}|+|\mathcal{G}| \leq\binom{n}{2}-\binom{n-4}{1}+1.
$
\end{itemize}
\end{lemma}
\begin{proof}
By Lemma \ref{Ln}, we may assume that $\mathcal{F}=\mathcal{L}(n, 3,|\mathcal{F}|)$ and $\mathcal{G}=\mathcal{L}(n, 2,|\mathcal{G}|)$.
If $1 \leq |\mathcal{F}| \leq n-2$, then
$
\{\{1,2,3\}\}\subseteq\mathcal{F}\subseteq \left\{\{1,2, i\}: 3 \leq i \leq n\right\}.
$
Thus, $\mathcal{F}$ is $2$-intersecting, and the result follows from Lemma \ref{F16}.
If $|\mathcal{F}| = n-1$, then
$
\mathcal{F}= \left\{\{1,2, i\}: 3 \leq i \leq n\right\}\cup \left\{\{1,3, 4\}\right\}.
$
Since $\mathcal{F}$ and $\mathcal{G}$ are cross-intersecting, we have
$
\mathcal{G}\subseteq \left\{\{1, i\}: 2 \leq i \leq n\right\}\cup \left\{\{2,3\}, \{2,4\}\right\}.
$
It follows that
$
|\mathcal{F}|+|\mathcal{G}| \leq 2n<\binom{n}{2}-\binom{n-3}{2}+1.
$
If $n \leq |\mathcal{F}| \leq 2n-5$, then
$
\left\{\{1,2, i\}: 3 \leq i \leq n\right\}\cup \left\{\{1,3, 4\}, \{1,3, 5\}\right\}\subseteq\mathcal{F}$ and $\mathcal{F}\subseteq \left\{\{1,2, i\}: 3 \leq i \leq n\right\}\cup \left\{\{1,3, i\}: 4 \leq i \leq n\right\}.
$
Since $\mathcal{F}$ and $\mathcal{G}$ are cross-intersecting, we infer that
$
\mathcal{G}\subseteq \left\{\{1, i\}: 2 \leq i \leq n\right\}\cup \left\{\{2,3\}\right\}.
$
Consequently,
$
|\mathcal{F}|+|\mathcal{G}| \leq 3n-5=\binom{n}{2}-\binom{n-3}{2}+1.
$
If $2n-4 \leq |\mathcal{F}| \leq 3n-9$, then
$
\left\{\{1,2, i\}: 3 \leq i \leq n\right\}\cup \left\{\{1,3, i\}: 4 \leq i \leq n\right\}\cup \left\{\{1,4, 5\}\right\}\subseteq\mathcal{F}
$
and
$
\mathcal{F}\subseteq \left\{\{1,2, i\}: 3 \leq i \leq n\right\}\cup \left\{\{1,3, i\}: 4 \leq i \leq n\right\}\cup \left\{\{1,4, i\}: 5 \leq i \leq n\right\}.
$
This implies that
$
\mathcal{G}\subseteq \left\{\{1, i\}: 2 \leq i \leq n\right\}.
$
Thus, we have
$
|\mathcal{F}|+|\mathcal{G}| \leq 4n-10.
$
For $n \geq 6$, let
$
f(n)=\binom{n}{2}-\binom{n-4}{1}+1-\left(4n-10\right)=\frac{1}{2}\left(n^2-11n+30 \right).
$
Since $f(6)= 0$, we have $f(n)\geq 0$ for $n \geq 6$. Hence,
$
|\mathcal{F}|+|\mathcal{G}| \leq \binom{n}{2}-\binom{n-4}{1}+1.
$
\end{proof}

\section{Proof of Theorem \ref{ma1}}\label{se3}

\subsection{Structure of the proof and new ingredients}

The proof of Theorem \ref{ma1} proceeds by a down-shift argument, following the framework introduced by Frankl \cite{F172}.
Recall that after a suitable translation, we may assume $|\mathcal{F}(i)| \le |\mathcal{F}(\bar i)|$ for all $i \in [n]$, and that  $\mathcal{F}$ is not contained in any translation of $ \mathcal{K}(n, s)$, $\mathcal{H}(n, s)$, or any  of the exceptional families depending on $s$: $\mathcal{R}(n, 2d)$ for $s=2d$; $\mathcal{T}(n, 5)$ for $s=5$; and $\mathcal{H}^*(n, 4)$, $\mathcal{R}^*(n, 4)$, $\mathcal{U}^*(n, 4)$
 for $s=4$.
Starting from such a family  $\mathcal{F}$ with $\Delta(\mathcal{F}) \le s$, we repeatedly apply down-shifts as described in Subsection \ref{subse22}. Depending on the outcome, we distinguish three main parts:

\begin{description}
\item[Part I (Section \ref{se42}).] 
The down-shift process terminates at a complex family $\mathcal{C}$. In this case, $\mathcal{C}$ is $s$-union, and the desired bounds follow directly from the second stability for union families (Theorem \ref{L24}).

\item[Part II (Section \ref{se43}).] 
At some step, we obtain a family $\mathcal{G}$ such that $S_j(\mathcal{G})$ is contained in one of the general structures $\mathcal{K}(n,s)$, $\mathcal{H}(n,s)$,  $\mathcal{R}(n,2d)$. Recall that $\mathcal{G}$ denotes the family obtained just before the final down-shift that first falls into such a structure and it satisfies $|\mathcal{G}(i)| \le |\mathcal{G}(\bar i)|$ for all $i\in [n]$ by the properties of down-shifts. 
This part is further divided into subcases depending on the target family and the parity of $s$. The key tools are Theorem \ref{H17}  and Lemmas \ref{H672}, \ref{W231} and \ref{HK32}.  Proposition \ref{prop1} allows us to assume $|\mathcal{G}(\bar{j})|$ is sufficiently  large for all $j\in [n]$.

\item[Part III (Section \ref{se44}).] 
The down-shift process leads to a family $\mathcal{G}$ with $S_j(\mathcal{G})$ contained in one of the special structures that appear only for small diameters, i.e., $\mathcal{T}(n,5)$ (when $s=5$) and $\mathcal{H}^*(n,4)$, $\mathcal{R}^*(n,4)$, $\mathcal{U}^*(n,4)$ (when $s=4$). These cases require a more detailed case analysis, using Theorem \ref{H17},  Lemmas \ref{F16} and \ref{Last} for $s=5$, and Lemmas  \ref{H672} and \ref{HK32} together with structural constraints for $s=4$.
\end{description}

While the overall strategy follows Frankl's first stability argument \cite{F172}, our proof requires several new ingredients that go substantially beyond the earlier work.
The first difference is the proliferation of extremal structures. Frankl's first stability only deals with $\mathcal{K}(n,s)$. In contrast, the second stability must additionally handle $\mathcal{H}(n,s), \mathcal{R}(n,2d)$, as well as several exceptional families that appear only for small diameters: $\mathcal{H}^*(n,4)$, $\mathcal{R}^*(n,4)$, $\mathcal{U}^*(n,4)$ when $s=4$, and $\mathcal{T}(n,5)$ when $s=5$. Each requires a separate case study, with careful tracking of the special elements $\{1,2,y\}$ and the down-shift index $j$.
The second challenge arises from the increased depth of the families. While Frankl's proof considers maximum set sizes $d+1$ (for even $s$) and $d+1,d+2$ (for odd $s$), we must also handle size $d+3$ for even $s$, leading to more complex cross-intersecting relations between layers. These are addressed by the refined inequalities in Lemmas \ref{H672}, \ref{W231} and \ref{HK32}.
In summary, our proof extends Frankl's method by handling more exceptional structures, deeper families, and more intricate cross-intersecting configurations.

\subsection{Part I: Down-shift ends with a complex}\label{se42}

Suppose that $\mathcal{F}$ is a complex, or $\mathcal{F}$ is not a complex but there is a complex $\mathcal{C}$ obtained from $\mathcal{F}$ by repeated down-shifts  satisfying $\mathcal{C}\nsubseteq\mathcal{K}(n, s),\mathcal{H}(n, s)$, and $\mathcal{C}\nsubseteq\mathcal{T}(n, 5)$ if $s=5$, in addition, $\mathcal{C}\nsubseteq\mathcal{R}(n, 2d)$ if $s=2d$, $\mathcal{C}\nsubseteq\mathcal{H}^*(n, 4), \mathcal{R}^*(n, 4),\mathcal{U}^*(n, 4)$ if $s=4$.
By Lemma \ref{K661}, we infer that
$\mathcal{F}$ or $\mathcal{C}$ is $s$-union and $|\mathcal{F}|=|\mathcal{C}|$. For $4 \leq s \leq n-2$,  Theorem \ref{L24} implies that
\begin{align*}
|\mathcal{F}| \leq \sum_{i=0}^d\binom{n}{i}-\binom{n-d-1}{d}-\binom{n-d-2}{d-1}+2
\end{align*}
for $s=2d$, and
\begin{align*}
|\mathcal{F}| \leq \sum_{i=0}^d\binom{n}{i}+\binom{n-1}{d}-\binom{n-d-2}{d}-\binom{n-d-3}{d-1}+2
\end{align*}
for  $s=2d+1$. For $s=3$, note that $\mathcal{F}, \mathcal{C} \nsubseteq \mathcal{K}(n, 3), \mathcal{H}(n, 3)$. The $3$-union property implies that  $|\mathcal{F}|$ is maximized by
$2^{\{a,b,c\}}$. Thus, we have $|\mathcal{F}|\leq8$.

\subsection{Part II: Down-shift lands in general structures}\label{se43}
In this part, we consider the cases where after repeated down-shifts, we obtain a family $\mathcal{G}$ such that $S_j(\mathcal{G})$ is contained in one of the general structures $\mathcal{K}(n,s)$, $\mathcal{H}(n,s)$,  $\mathcal{R}(n,2d)$. Recall that $\mathcal{G}$ satisfies $|\mathcal{G}|=|\mathcal{F}|, \Delta(\mathcal{G}) \leq s$ and 
$|\mathcal{G}(i)| \leq |\mathcal{G}(\bar{i})|$ for all $i \in [n]$.

\subsubsection{Down-shift lands in $\mathcal{K}(n,2d)$}\label{ssec31}
 Suppose that  $s=2d$, $\mathcal{F}$ is not a complex and
there is a family $\mathcal{G}$ obtained from $\mathcal{F}$ by repeated down-shifts  satisfying $\mathcal{G}\nsubseteq\mathcal{K}(n, 2d),\mathcal{H}(n, 2d),\mathcal{R}(n, 2d)$, in addition, $\mathcal{G}\nsubseteq\mathcal{H}^*(n, 4), \mathcal{R}^*(n, 4), \mathcal{U}^*(n, 4)$ if $s=4$, but $S_j(\mathcal{G})\subseteq\mathcal{K}(n, 2d)$ for some $j \in [n]$.
Then ${\rm max}\{|G|: G\in \mathcal{G}\}=d+1$, and for any $G\in \mathcal{G}_{d+1}$, we have $j\in G$,  $G\backslash \{j\}\notin \mathcal{G}$ and $G\backslash \{j\}\in S_j(\mathcal{G})$. Consequently, $\mathcal{G}_{d+1}(j)\neq \emptyset,\ \mathcal{G}_{d+1}(\bar{j})= \emptyset$ and $\mathcal{G}_{d+1}(j)\cap \mathcal{G}_{d}(\bar{j})=\emptyset$.
Since $\Delta(\mathcal{G})\leq 2d$, the families $\mathcal{G}_{d+1}(j)$ and $\mathcal{G}_{d}(\bar{j})$ are cross-intersecting. Moreover, the condition $\mathcal{G}\nsubseteq\mathcal{H}(n, 2d)$ implies  $|\mathcal{G}_{d+1}(j)|\geq 2$.

If $|\mathcal{G}_{d}(\bar{j})|\leq 1$, then
$|\mathcal{G}(\bar{j})|=\sum_{i=0}^d|\mathcal{G}_i(\bar{j})|\leq \sum_{i=0}^{d-1}\binom{n-1}{i}+1$. In view of Proposition \ref{prop1}, the result follows.

If $|\mathcal{G}_{d}(\bar{j})|\geq 2$, note that $d\geq 2$ and $n\geq 2d+2$. Applying Lemmas \ref{W231} and \ref{HK32} to $\mathcal{G}_d(\bar{j})$ and $\mathcal{G}_{d+1}(j)$ yields
$$
|\mathcal{G}_d(\bar{j})|+|\mathcal{G}_{d+1}(j)|
\leq\binom{n-1}{d}-\binom{n-d-1}{d}-\binom{n-d-2}{d-1}+1.
$$
Since $|\mathcal{F}|= \sum_{i=0}^{d+1}|\mathcal{G}_i|$ and $|\mathcal{G}_d(j)|\leq \binom{n-1}{d-1}$,  it follows that
\begin{align}\label{c0}
\begin{split}
|\mathcal{F}|\leq &\sum_{i=0}^{d-1}\binom{n}{i}+|\mathcal{G}_d(j)|+|\mathcal{G}_d(\bar{j})|+|\mathcal{G}_{d+1}(j)|\\
\leq&\sum_{i=0}^{d}\binom{n}{i}-\binom{n-d-1}{d}-\binom{n-d-2}{d-1}+1\\
<&\sum_{i=0}^{d}\binom{n}{i}-\binom{n-d-1}{d}-\binom{n-d-2}{d-1}+2.
\end{split}
\end{align}

\subsubsection{Down-shift lands in $\mathcal{K}(n,2d+1)$}\label{ssec32}
 Suppose that  $s=2d+1$, $\mathcal{F}$ is not a complex and
there is a family $\mathcal{G}$ obtained from $\mathcal{F}$ by repeated down-shifts  satisfying $\mathcal{G}\nsubseteq\mathcal{K}(n, 2d+1), \mathcal{H}(n, 2d+1)$,  in addition, $\mathcal{G}\nsubseteq\mathcal{T}(n, 5)$ if $s=5$, but $S_j(\mathcal{G})\subseteq\mathcal{K}(n, 2d+1)$ for some $j \in [n]$.
Since $\mathcal{G}\nsubseteq\mathcal{K}(n, 2d+1)$, we have
${\rm max}\{|G|: G\in \mathcal{G}\}\geq d+1$. 
Our proof naturally falls into two cases.

\paragraph{${\rm max}\{|G|: G\in \mathcal{G}\}=d+1$}

Since $\Delta(\mathcal{G}) \leq 2d+1$, the family $\mathcal{G}_{d+1}$ is intersecting. Clearly, $\mathcal{G}_{d+1}$ is not EKR otherwise $\mathcal{G}\subseteq \mathcal{K}(n, 2d+1)$.
Moreover,  $\mathcal{G}_{d+1}$ is not HM otherwise $\mathcal{G}\subseteq \mathcal{H}(n, 2d+1)$, and if $d=2, \mathcal{G}_{3} \nsubseteq \mathcal{T}(n, 3)$ otherwise $\mathcal{G}\subseteq \mathcal{T}(n, 5)$.  For $d=1$, the case is trivial because there is no such $\mathcal{G}_{2}$.
For $d\geq 2$,  applying Theorem \ref{H17} to $\mathcal{G}_{d+1}$ gives
$$
|\mathcal{G}_{d+1}|\leq\binom{n-1}{d}-\binom{n-d-2}{d}-\binom{n-d-3}{d-1}+2.
$$
Consequently, we obtain the  following bound on $|\mathcal{F}|$:
\begin{align*}
|\mathcal{F}|\leq \sum_{i=0}^d\binom{n}{i}+\binom{n-1}{d}-\binom{n-d-2}{d}-\binom{n-d-3}{d-1}+2.
\end{align*}

\paragraph{${\rm max}\{|G|: G\in \mathcal{G}\}\geq d+2$}

Since $S_j(\mathcal{G})\subseteq\mathcal{K}(n, 2d+1)$, we infer that ${\rm max}\{|G|: G\in \mathcal{G}\}= d+2$ and $|\mathcal{G}_{d+1}(\bar{y},\bar{j})|=0$ with $j\neq y$. For any $G\in \mathcal{G}_{d+2}$, we have $j, y\in G$,  $G\backslash \{j\}\notin \mathcal{G}$ and $G\backslash \{j\}\in S_j(\mathcal{G})$. Hence, $|\mathcal{G}_{d+2}|=|\mathcal{G}_{d+2}(y,j)|\neq 0$. It follows that $\mathcal{G}_{d+2}(y,j)\cap \mathcal{G}_{d+1}(y,\bar{j})=\emptyset$ otherwise there exists $G\in \mathcal{G}_{d+2}$ satisfying $G\backslash \{j\}\in \mathcal{G}_{d+1}$. Moreover, $\mathcal{G}_{d+1}(\bar{y},j)\cap \mathcal{G}_{d}(\bar{y},\bar{j})=\emptyset$ otherwise there exists $y\notin G\in \mathcal{G}_{d+1}$ satisfying $G\backslash \{j\}\in \mathcal{G}_{d}$ and then $y\notin S_j(G)=G$, which implies $S_j(\mathcal{G})\nsubseteq\mathcal{K}(n, 2d+1)$, a contradiction. From $\Delta(\mathcal{G})\leq 2d+1$, we also have that
$\mathcal{G}_{d+2}(y,j)$ and  $\mathcal{G}_{d}(\bar{y},\bar{j})$ are cross-intersecting, and  $\mathcal{G}_{d+1}(y,\bar{j})$ and  $\mathcal{G}_{d+1}(\bar{y},j)$ are cross-intersecting.
To formalize these observations, we divide the proof into two lemmas.

\begin{lemma}\label{c3.1}
Suppose that ${\rm max}\{|G|: G\in \mathcal{G}\}=d+2$ and $\mathcal{G}_{d}(\bar{y},\bar{j})\neq \emptyset$. Then
\begin{align*}
|\mathcal{F}|\leq \sum_{i=0}^d\binom{n}{i}+\binom{n-1}{d}-2\binom{n-d-2}{d}+2.
\end{align*}
In particular, if $d=1$, then $|\mathcal{F}|\leq 8$.
\end{lemma}

\begin{proof}
Since $\mathcal{G}_{d+2}(y, j)\neq \emptyset$ and $n\geq 2d+3$,  applying Lemma \ref{H672} yields
\begin{align}\label{c1}
|\mathcal{G}_{d+2}(y,j)|+|\mathcal{G}_{d}(\bar{y},\bar{j})|\leq\binom{n-2}{d}-\binom{n-d-2}{d}+1.
\end{align}
We also have the following elementary bounds:
\begin{align}\label{c2}
|\mathcal{G}_{d+1}(\bar{y},\bar{j})|=0,~ |\mathcal{G}_{d+1}(y,j)|\leq \binom{n-2}{d-1}, ~ |\mathcal{G}_{d}(j)|\leq \binom{n-1}{d-1}, ~ |\mathcal{G}_{d}(y,\bar{j})|\leq \binom{n-2}{d-1}.
\end{align}
We now consider three cases.

{\bf Case 1. Assume that $\mathcal{G}_{d+1}(y,\bar{j})\neq \emptyset$ and  $\mathcal{G}_{d+1}(\bar{y},j)\neq \emptyset$.}
Applying Lemma \ref{H672} again  yields
\begin{align}\label{c4}
|\mathcal{G}_{d+1}(y,\bar{j})|+|\mathcal{G}_{d+1}(\bar{y},j)|\leq\binom{n-2}{d}-\binom{n-d-2}{d}+1.
\end{align}
Combining this with (\ref{c1}) and (\ref{c2}), we obtain
\begin{align}\label{ct4.1}
\begin{split}
|\mathcal{F}|\leq& \sum_{i=0}^{d-1}\binom{n}{i}+|\mathcal{G}_d(j)|+|\mathcal{G}_d(y,\bar{j})|+|\mathcal{G}_{d+2}(y,j)|+|\mathcal{G}_{d}(\bar{y},\bar{j})|\\
&+|\mathcal{G}_{d+1}(y,j)|+|\mathcal{G}_{d+1}(y,\bar{j})|+|\mathcal{G}_{d+1}(\bar{y},j)|\\
\leq&\sum_{i=0}^{d}\binom{n}{i}+\binom{n-1}{d}-2\binom{n-d-2}{d}+2.
\end{split}
\end{align}

{\bf Case 2. Assume that $\mathcal{G}_{d+1}(y,\bar{j})=\emptyset$.}
Since $|\mathcal{G}_{d+1}(\bar{y},\bar{j})|=0$,  we have $\mathcal{G}_{d+1}(\bar{j})=\emptyset$. From (\ref{c1}), we get
$
|\mathcal{G}_{d}(\bar{y},\bar{j})|\leq\binom{n-2}{d}-\binom{n-d-2}{d}.
$
Combining this with (\ref{c2}), we obtain
\begin{align}\label{ca1}
|\mathcal{G}_{d}(\bar{j})|\leq\binom{n-1}{d}-\binom{n-d-2}{d}.
\end{align}
Recall that $\mathcal{G}_{d+2}(\bar{j})=\emptyset$. It follows that
\begin{align}\label{c3}
|\mathcal{G}(\bar{j})|=\sum_{i=0}^{d-1}|\mathcal{G}_i(\bar{j})|+|\mathcal{G}_{d}(\bar{j})| \leq \sum_{i=0}^{d-1}\binom{n-1}{i}+\binom{n-1}{d}-\binom{n-d-2}{d}.
\end{align}
In view of Proposition \ref{prop1},  Lemma \ref{c3.1} holds.

{\bf Case 3. Assume that $\mathcal{G}_{d+1}(\bar{y},j)=\emptyset$.}
Since $|\mathcal{G}_{d+1}(\bar{y},\bar{j})|=0$, we have  $\mathcal{G}_{d+1}(\bar{y})=\emptyset$. Observe that we still have
$
|\mathcal{G}_{d}(\bar{y},\bar{j})|\leq\binom{n-2}{d}-\binom{n-d-2}{d}.
$
This, together with $|\mathcal{G}_{d}(\bar{y},j)|\leq \binom{n-2}{d-1}$, implies 
\begin{align}\label{ca2}
|\mathcal{G}_{d}(\bar{y})|\leq\binom{n-1}{d}-\binom{n-d-2}{d}.
\end{align}
Since $\mathcal{G}_{d+2}(\bar{y})=\emptyset$, the same argument as (\ref{c3}) yields 
$
|\mathcal{G}(\bar{y})|
\leq \sum_{i=0}^{d-1}\binom{n-1}{i}+\binom{n-1}{d}-\binom{n-d-2}{d},
$
and the result follows from  Proposition \ref{prop1}.
\end{proof}

\begin{remark}
In Lemma \ref{c3.1}, for $d=1$, $|\mathcal{F}|$ can be maximized by a translation of the following family:
$
\{\emptyset, \{j\}, \{y\}, \{x_0\}, \{y, j\}, \{y, x_1\}, \{j, x_1\}, \{y, j, x_0\}\}.
$ 
Here $j,y,x_0,x_1$ are distinct elements of $[n]$.
\end{remark}

\begin{lemma}\label{c3.2}
Suppose that ${\rm max}\{|G|: G\in \mathcal{G}\}=d+2$ and $\mathcal{G}_{d}(\bar{y},\bar{j})= \emptyset$. Then
\begin{align*}
|\mathcal{F}|\leq \sum_{i=0}^d\binom{n}{i}+\binom{n-1}{d}-2\binom{n-d-2}{d}.
\end{align*}
In particular, if $d=1$, then $|\mathcal{F}|\leq 6$.
\end{lemma}
\begin{proof}

If  $\mathcal{G}_{d+1}(y,\bar{j})\neq \emptyset$ and  $\mathcal{G}_{d+1}(\bar{y},j)\neq \emptyset$, then (\ref{c4}) holds. That is
$$
|\mathcal{G}_{d+1}(y,\bar{j})|+|\mathcal{G}_{d+1}(\bar{y},j)|\leq\binom{n-2}{d}-\binom{n-d-2}{d}+1.
$$
This, together with $|\mathcal{G}_{d+1}(\bar{y},\bar{j})|=0$, implies
$
|\mathcal{G}_{d+1}(\bar{j})|\leq\binom{n-2}{d}-\binom{n-d-2}{d}.
$
Moreover, we have
$
|\mathcal{G}_{d}(\bar{j})|=|\mathcal{G}_{d}(y,\bar{j})|\leq\binom{n-2}{d-1},~ |\mathcal{G}_{d+2}(\bar{j})|=0.
$
Consequently,
\begin{align*}
|\mathcal{G}(\bar{j})|=\sum_{i=0}^{d-1}|\mathcal{G}_i(\bar{j})|+|\mathcal{G}_{d}(\bar{j})|+|\mathcal{G}_{d+1}(\bar{j})|\leq \sum_{i=0}^{d-1}\binom{n-1}{i}+ \binom{n-1}{d}-\binom{n-d-2}{d}.
\end{align*}
In view of Proposition \ref{prop1},  Lemma \ref{c3.2} holds.

   Next we consider the case $\mathcal{G}_{d+1}(y,\bar{j})=\emptyset$. Since $|\mathcal{G}_{d+1}(\bar{y},\bar{j})|=0$, we have $\mathcal{G}_{d+1}(\bar{j})=\emptyset$. Combining this with  $\mathcal{G}_{d+2}(\bar{j})=\emptyset$ and  $|\mathcal{G}_{d}(\bar{j})|= |\mathcal{G}_{d}(y,\bar{j})|$, we obtain
\begin{align}\label{c5}
|\mathcal{G}(\bar{j})|=\sum_{i=0}^{d-1}|\mathcal{G}_i(\bar{j})|+|\mathcal{G}_{d}(y,\bar{j})| \leq \sum_{i=0}^{d-1}\binom{n-1}{i}+ \binom{n-2}{d-1}.
\end{align}
In view of Proposition \ref{prop1} again,  Lemma \ref{c3.2} holds.

  It remains to consider the case $\mathcal{G}_{d+1}(\bar{y},j)=\emptyset$.
Since $|\mathcal{G}_{d+1}(\bar{y},\bar{j})|=0$, we have $\mathcal{G}_{d+1}(\bar{y})=\emptyset$. Since $\mathcal{G}_{d+2}(\bar{y})=\emptyset$ and  $|\mathcal{G}_{d}(\bar{y})|= |\mathcal{G}_{d}(\bar{y},j)|$, the same argument as (\ref{c5}) yields
$
|\mathcal{G}(\bar{y})|\leq \sum_{i=0}^{d-1}\binom{n-1}{i}+ \binom{n-2}{d-1},
$
and the result follows from  Proposition \ref{prop1}.
\end{proof}

Observe that
$$
\sum_{i=0}^d\binom{n}{i}+\binom{n-1}{d}-2\binom{n-d-2}{d}+2
< \sum_{i=0}^d\binom{n}{i}+\binom{n-1}{d}-\binom{n-d-2}{d}-\binom{n-d-3}{d-1}+2.
$$
By Lemmas \ref{c3.1} and \ref{c3.2}, the desired bound holds  in the case ${\rm max}\{|G|: G\in \mathcal{G}\}\geq d+2$.

\subsubsection{Down-shift lands in $\mathcal{H}(n,2d)$}\label{ssec33}

 Suppose that  $s=2d$, $\mathcal{F}$ is not a complex and
there is a family $\mathcal{G}$ obtained from $\mathcal{F}$ by repeated down-shifts  satisfying $\mathcal{G}\nsubseteq\mathcal{K}(n, 2d), \mathcal{H}(n, 2d), \mathcal{R}(n, 2d)$, in addition, $\mathcal{G}\nsubseteq\mathcal{H}^*(n, 4), \mathcal{R}^*(n, 4), \mathcal{U}^*(n, 4)$ if $s=4$, but $S_j(\mathcal{G})\subseteq\mathcal{H}(n, 2d)$ for some $j \in [n]$.
By Subsubsection \ref{ssec31}, we may assume that $S_j(\mathcal{G})\nsubseteq\mathcal{K}(n, 2d)$. Then we have ${\rm max}\{|G|: G\in \mathcal{G}\}\geq d+1$.

\paragraph{${\rm max}\{|G|: G\in \mathcal{G}\}=d+1$}

Since $S_j(\mathcal{G})\subseteq\mathcal{H}(n, 2d)$, we have $|\mathcal{G}_{d+1}(\bar{j})|\leq 1$.
We first consider the case  $|\mathcal{G}_{d+1}(\bar{j})|=0$.

\begin{lemma}\label{c3.3}
Suppose that ${\rm max}\{|G|: G\in \mathcal{G}\}=d+1$ and $|\mathcal{G}_{d+1}(\bar{j})|=0$. Then
\begin{align*}
|\mathcal{F}|\leq \sum_{i=0}^d\binom{n}{i}-\binom{n-d-1}{d}-\binom{n-d-2}{d-1}+1.
\end{align*}
\end{lemma}
\begin{proof}
Since $S_j(\mathcal{G})\subseteq\mathcal{H}(n, 2d)$ and $S_j(\mathcal{G})\nsubseteq\mathcal{K}(n, 2d)$, we have $|\mathcal{G}_{d+1}(j)\cap \mathcal{G}_{d}(\bar{j})|=1$.  Moreover, $\mathcal{G}\nsubseteq\mathcal{H}(n, 2d)$ implies $|\mathcal{G}_{d+1}(j)|\geq 2$. In view of $\Delta(\mathcal{G})\leq 2d$,  $\mathcal{G}_{d+1}(j)$ and $\mathcal{G}_{d}(\bar{j})$ are cross-intersecting.

If $|\mathcal{G}_{d}(\bar{j})|=1$, then
$
|\mathcal{G}(\bar{j})|=\sum_{i=0}^d|\mathcal{G}_i(\bar{j})|\leq \sum_{i=0}^{d-1}\binom{n-1}{i}+1.
$
By Proposition \ref{prop1}, we have
$$
|\mathcal{F}| \leq \sum_{i=0}^{d-1}\binom{n}{i}+\binom{n-1}{d-1}+4<\sum_{i=0}^d\binom{n}{i}-\binom{n-d-1}{d}-\binom{n-d-2}{d-1}+1.
$$

If $|\mathcal{G}_{d}(\bar{j})|\geq 2$, then by Lemmas \ref{W231} and \ref{HK32}, 
$$
|\mathcal{G}_d(\bar{j})|+|\mathcal{G}_{d+1}(j)|
\leq\binom{n-1}{d}-\binom{n-d-1}{d}-\binom{n-d-2}{d-1}+1.
$$
Combining this with (\ref{c0}) yields
$
|\mathcal{F}|\leq\sum_{i=0}^d\binom{n}{i}-\binom{n-d-1}{d}-\binom{n-d-2}{d-1}+1.
$
\end{proof}

Next we consider the case  $|\mathcal{G}_{d+1}(\bar{j})|=1$.

\begin{lemma}\label{c3.4}
Suppose that  ${\rm max}\{|G|: G\in \mathcal{G}\}=d+1$ and  $|\mathcal{G}_{d+1}(\bar{j})|=1$. Then
\begin{align*}
|\mathcal{F}|\leq \sum_{i=0}^d\binom{n}{i}-\binom{n-d-1}{d}-\binom{n-d-2}{d-1}+2.
\end{align*}
\end{lemma}
\begin{proof}
Since $S_j(\mathcal{G})\subseteq\mathcal{H}(n, 2d)$ and $S_j(\mathcal{G})\nsubseteq\mathcal{K}(n, 2d)$, we have $|\mathcal{G}_{d+1}(j)\cap \mathcal{G}_{d}(\bar{j})|=0$.  Furthermore, $\mathcal{G}\nsubseteq\mathcal{H}(n, 2d)$ implies that $|\mathcal{G}_{d+1}(j)|\geq 1$. Since  $\Delta(\mathcal{G})\leq 2d$,  we  infer that $\mathcal{G}_{d+1}(j)$ and $\mathcal{G}_{d}(\bar{j})$ are cross-intersecting, and $\mathcal{G}_{d+1}(\bar{j})$ and $\mathcal{G}_{d+1}(j)$ are cross-intersecting.

If $|\mathcal{G}_{d}(\bar{j})|\leq 1$, then
$
|\mathcal{G}(\bar{j})|=\sum_{i=0}^{d+1}|\mathcal{G}_i(\bar{j})|\leq  \sum_{i=0}^{d-1}\binom{n-1}{i}+2.
$
In view of Proposition \ref{prop1}, Lemma \ref{c3.4} holds.

 If $|\mathcal{G}_{d+1}(j)|= 1$, then $\mathcal{G}_{d+1}= \{A, B\}$ for some $x\in A, B$ with  $x\neq j, x\in [n]$ because $\mathcal{G}_{d+1}(\bar{j})$ and $\mathcal{G}_{d+1}(j)$ are cross-intersecting. Hence,  $|\mathcal{G}_{d+1}(\bar{x})|=0$.  Note that $\mathcal{G}_{d+1}(x)$ and $\mathcal{G}_{d}(\bar{x})$ are cross-intersecting. 
 If $|\mathcal{G}_{d+1}(x)\cap \mathcal{G}_{d}(\bar{x})|=0$, then $S_x(\mathcal{G})\subseteq\mathcal{K}(n, 2d)$. By Subsubsection \ref{ssec31},  Lemma \ref{c3.4} holds.
 If $|\mathcal{G}_{d+1}(x)\cap \mathcal{G}_{d}(\bar{x})|=1$, then the same argument with Lemma \ref{c3.3} works.
 If  $|\mathcal{G}_{d+1}(x)\cap \mathcal{G}_{d}(\bar{x})|=2$, then applying Lemmas \ref{W231} and \ref{HK32} yields
$$
|\mathcal{G}_d(\bar{x})|+|\mathcal{G}_{d+1}(x)|
\leq\binom{n-1}{d}-\binom{n-d-1}{d}-\binom{n-d-2}{d-1}+2.
$$
It follows that
\begin{align*}
|\mathcal{F}|=& \sum_{i=0}^{d+1}|\mathcal{G}_i|= \sum_{i=0}^{d-1}|\mathcal{G}_i|+|\mathcal{G}_d(x)|+|\mathcal{G}_d(\bar{x})|+|\mathcal{G}_{d+1}(x)|\\
\leq&\sum_{i=0}^d\binom{n}{i}-\binom{n-d-1}{d}-\binom{n-d-2}{d-1}+2.
\end{align*}

      It remains to consider the case  $|\mathcal{G}_{d}(\bar{j})|\geq 2$ and $|\mathcal{G}_{d+1}(j)|\geq 2$. Since $|\mathcal{G}_{d+1}(j)\cap \mathcal{G}_{d}(\bar{j})|=0$, applying Lemmas \ref{W231} and \ref{HK32} yields
$$
|\mathcal{G}_d(\bar{j})|+|\mathcal{G}_{d+1}(j)|
\leq\binom{n-1}{d}-\binom{n-d-1}{d}-\binom{n-d-2}{d-1}+1.
$$
Then the same argument as (\ref{c0}) yields the desired result.
\end{proof}

By Lemmas \ref{c3.3} and \ref{c3.4}, the desired bound holds  in the case ${\rm max}\{|G|: G\in \mathcal{G}\}= d+1$.

\paragraph{${\rm max}\{|G|: G\in \mathcal{G}\}\geq d+2$}

In this case, $S_j(\mathcal{G})\subseteq\mathcal{H}(n, 2d)$ implies  ${\rm max}\{|G|: G\in \mathcal{G}\}= d+2$. Moreover, we have  $\mathcal{G}_{d+2}=\{G_0\}$ with $j\in G_0$ and
$G_0\backslash \{j\}\notin \mathcal{G}_{d+1}$. Clearly,  $\mathcal{G}_{d+1}\subseteq\{G: j\in G, |G|=d+1\}$ and $\mathcal{G}_{d+1}(j)\cap \mathcal{G}_{d}(\bar{j})=\emptyset$. Observe that $\mathcal{G}_{d+1}(j)$ and $\mathcal{G}_{d}(\bar{j})$ are cross-intersecting. Furthermore, $\mathcal{G}_{d+2}$ and $\mathcal{G}_{d}(\bar{j})$ are cross-intersecting, and $\mathcal{G}_{d+2}$ and $\mathcal{G}_{d-1}$ are cross-intersecting. Consequently,
\begin{align}\label{c7}
|\mathcal{G}_d(\bar{j})|\leq\binom{n-1}{d}-\binom{n-d-2}{d},\quad |\mathcal{G}_{d-1}|\leq\binom{n}{d-1}-\binom{n-d-2}{d-1}.
\end{align}
The proof  follows from the lemma below.

\begin{lemma}\label{c3.5}
Suppose that  ${\rm max}\{|G|: G\in \mathcal{G}\}=d+2$. Then
\begin{align*}
|\mathcal{F}|\leq \sum_{i=0}^d\binom{n}{i}-\binom{n-d-1}{d}-\binom{n-d-2}{d-1}+1.
\end{align*}
\end{lemma}

\begin{proof} We consider four cases.

 If $|\mathcal{G}_{d}(\bar{j})|\leq 1$, then
$
|\mathcal{G}(\bar{j})|=\sum_{i=0}^d|\mathcal{G}_i(\bar{j})|\leq \sum_{i=0}^{d-1}\binom{n-1}{i}+1,
$
and the result follows from  Proposition \ref{prop1}.

If $|\mathcal{G}_{d+1}(j)|=0$, then $\mathcal{G}\subseteq\mathcal{R}(n, 2d)$, a contradiction.

 If $|\mathcal{G}_{d+1}(j)|=1$, note that $\mathcal{G}_{d+1}(j)$ and $\mathcal{G}_{d}(\bar{j})$ are cross-intersecting and $\mathcal{G}_{d+1}(j)\cap \mathcal{G}_{d}(\bar{j})=\emptyset$. By Lemma \ref{H672}, we have
$$
|\mathcal{G}_d(\bar{j})|+|\mathcal{G}_{d+1}(j)|
\leq\binom{n-1}{d}-\binom{n-d-1}{d}.
$$
Combining this with $|\mathcal{G}_{d-1}|\leq\binom{n}{d-1}-\binom{n-d-2}{d-1}$ and $|\mathcal{G}_{d}(j)|\leq \binom{n-1}{d-1}$,  we obtain
\begin{align*}
|\mathcal{F}|=&  \sum_{i=0}^{d-2}|\mathcal{G}_{i}|+|\mathcal{G}_{d-1}|+|\mathcal{G}_{d}(j)|+|\mathcal{G}_{d}(\bar{j})|+|\mathcal{G}_{d+1}(j)|+1\\
\leq&\sum_{i=0}^d\binom{n}{i}-\binom{n-d-1}{d}-\binom{n-d-2}{d-1}+1.
\end{align*}

  If $|\mathcal{G}_{d+1}(j)|\geq 2$ and $|\mathcal{G}_{d}(\bar{j})|\geq 2$, then applying Lemmas \ref{W231} and \ref{HK32} yields
$$
|\mathcal{G}_d(\bar{j})|+|\mathcal{G}_{d+1}(j)|
\leq\binom{n-1}{d}-\binom{n-d-1}{d}-\binom{n-d-2}{d-1}+1.
$$
This, together with (\ref{c7}), implies
\begin{align*}
|\mathcal{F}|=&  \sum_{i=0}^{d-2}|\mathcal{G}_{i}|+|\mathcal{G}_{d-1}|+|\mathcal{G}_{d}(j)|+|\mathcal{G}_{d}(\bar{j})|+|\mathcal{G}_{d+1}(j)|+1\\
\leq&\sum_{i=0}^d\binom{n}{i}-\binom{n-d-1}{d}-2\binom{n-d-2}{d-1}+2,
\end{align*}
which is smaller than the required upper bound. 
This completes the proof of Lemma \ref{c3.5}.
\end{proof}

\subsubsection{Down-shift lands in $\mathcal{H}(n,2d+1)$}

 Suppose that  $s=2d+1$, $\mathcal{F}$ is not a complex and
there is a family $\mathcal{G}$ obtained from $\mathcal{F}$ by repeated down-shifts  satisfying $\mathcal{G}\nsubseteq\mathcal{K}(n, 2d+1), \mathcal{H}(n, 2d+1)$,  in addition, $\mathcal{G}\nsubseteq\mathcal{T}(n, 5)$ if $s=5$, but $S_j(\mathcal{G})\subseteq\mathcal{H}(n, 2d+1)$ for some $j \in [n]$.
Since $\mathcal{G}\nsubseteq\mathcal{K}(n, 2d+1)$, we have
${\rm max}\{|G|: G\in \mathcal{G}\}\geq d+1$.
For the case ${\rm max}\{|G|: G\in \mathcal{G}\}= d+1$, the proof is simply the same as that of Subsubsection \ref{ssec32}.
Now assume that ${\rm max}\{|G|: G\in \mathcal{G}\}\geq d+2$. Then $S_j(\mathcal{G})\subseteq\mathcal{H}(n, 2d+1)$ implies that ${\rm max}\{|G|: G\in \mathcal{G}\}= d+2$ and $|\mathcal{G}_{d+1}(\bar{y},\bar{j})|\leq 1$.

If $j= y$, then $|\mathcal{G}_{d+2}|=|\mathcal{G}_{d+2}(y)|=1$ and $|\mathcal{G}_{d+1}|=|\mathcal{G}_{d+1}(y)|$. Observe that $\Delta(\mathcal{G})\leq 2d+1$ implies that
$\mathcal{G}_{d+2}(y)$ and  $\mathcal{G}_{d}(\bar{y})$ are cross-intersecting.
It follows that
$
|\mathcal{G}_{d}(\bar{y})|\leq\binom{n-1}{d}-\binom{n-d-2}{d}.
$
Consequently, 
\begin{align*}
|\mathcal{G}(\bar{y})|=\sum_{i=0}^{d-1}|\mathcal{G}_i(\bar{y})|+|\mathcal{G}_{d}(\bar{y})| \leq \sum_{i=0}^{d-1}\binom{n-1}{i}+\binom{n-1}{d}-\binom{n-d-2}{d}.
\end{align*}
Then  the result follows from Proposition \ref{prop1}.

From now on, let $j\neq y$.  For any $G\in \mathcal{G}_{d+2}$, we have $j\in G$,  $G\backslash \{j\}\notin \mathcal{G}$ and $G\backslash \{j\}\in S_j(\mathcal{G})$. Then $ \mathcal{G}_{d+2}(j)\cap \mathcal{G}_{d+1}(\bar{j})=\emptyset$.
 Moreover, $\Delta(\mathcal{G})\leq 2d+1$ implies that
$\mathcal{G}_{d+2}(y,j)$ and  $\mathcal{G}_{d}(\bar{y},\bar{j})$ are cross-intersecting, $\mathcal{G}_{d+1}(y,\bar{j})$ and  $\mathcal{G}_{d+1}(\bar{y},j)$ are cross-intersecting,  $\mathcal{G}_{d+2}(\bar{y},j)$ and  $\mathcal{G}_{d+1}(y,\bar{j})$ are cross-intersecting, $\mathcal{G}_{d+2}(\bar{y},j)$ and  $\mathcal{G}_{d}(\bar{j})$ are cross-intersecting.\vspace{2mm}

Recall that $|\mathcal{G}_{d+1}(\bar{y},\bar{j})|\leq 1$.
 \underline{Firstly, let us turn to the case  $|\mathcal{G}_{d+1}(\bar{y},\bar{j})|=0$}. {\bf We complete this case by proving the next two lemmas (Lemmas \ref{c3.6} and \ref{c3.7}).}

 If $|\mathcal{G}_{d+2}(\bar{y},j)|\neq 0$, then $|\mathcal{G}_{d+2}(\bar{y},j)|=1$ and $\mathcal{G}_{d+1}(\bar{y},j)\cap \mathcal{G}_{d}(\bar{y},\bar{j})=\emptyset$ otherwise $S_j(\mathcal{G})\nsubseteq\mathcal{H}(n, 2d+1)$.

\begin{lemma}\label{c3.6}
Suppose that $|\mathcal{G}_{d+1}(\bar{y},\bar{j})|=0$ and $|\mathcal{G}_{d+2}(\bar{y},j)|=1$.
Then
\begin{align*}
|\mathcal{F}|\leq \sum_{i=0}^d\binom{n}{i}+\binom{n-1}{d}-\binom{n-d-2}{d}-\binom{n-d-3}{d}+1.
\end{align*}
In particular, if $d=1$, then $|\mathcal{F}|\leq 8$.
\end{lemma}
\begin{proof}
We consider two cases.

{\bf Case 1. Assume that $\mathcal{G}_{d+2}(y, j)= \emptyset$.}
In this case, we have  $|\mathcal{G}_{d+2}|=|\mathcal{G}_{d+2}(\bar{y}, j)|=1$. Since $\mathcal{G}_{d+2}(\bar{y},j)$ and  $\mathcal{G}_{d}(\bar{j})$ are cross-intersecting, we obtain
\begin{align}\label{c8}
|\mathcal{G}_{d}(\bar{j})|\leq\binom{n-1}{d}-\binom{n-d-2}{d}.
\end{align}
Observe that
\begin{align}\label{c9}
|\mathcal{G}_{d+1}(y,j)|\leq \binom{n-2}{d-1}, \quad |\mathcal{G}_{d}(j)|\leq \binom{n-1}{d-1}.
\end{align}

If both $\mathcal{G}_{d+1}(y,\bar{j})$ and  $\mathcal{G}_{d+1}(\bar{y},j)$ are non-empty, then applying Lemma \ref{H672} yields
$$
|\mathcal{G}_{d+1}(y,\bar{j})|+|\mathcal{G}_{d+1}(\bar{y},j)|\leq\binom{n-2}{d}-\binom{n-d-2}{d}+1.
$$
This, together with (\ref{c8}) and (\ref{c9}), implies
\begin{align*}
|\mathcal{F}|=& \sum_{i=0}^{d-1}|\mathcal{G}_{i}|+|\mathcal{G}_d(j)|+|\mathcal{G}_d(\bar{j})|+|\mathcal{G}_{d+1}(y,j)|+|\mathcal{G}_{d+1}(y,\bar{j})|+|\mathcal{G}_{d+1}(\bar{y},j)|+1\\
\leq&\sum_{i=0}^d\binom{n}{i}+\binom{n-1}{d}-2\binom{n-d-2}{d}+2,
\end{align*}
which does not exceed the required upper bound.

If  $\mathcal{G}_{d+1}(y,\bar{j})=\emptyset$, then $\mathcal{G}_{d+1}(\bar{j})=\emptyset$. Combining this with $\mathcal{G}_{d+2}(\bar{j})=\emptyset$ and (\ref{c8}), applying the same argument as (\ref{c3})  yields
$$
|\mathcal{G}(\bar{j})|
\leq\sum_{i=0}^{d-1}\binom{n-1}{i}+\binom{n-1}{d}-\binom{n-d-2}{d}.
$$
By Proposition \ref{prop1}, we have
$$
|\mathcal{F}| \leq \sum_{i=0}^{d} \binom{n}{i} +\binom{n-1}{d}-2\binom{n-d-2}{d},
$$
which is smaller than the required upper bound.

  If  $\mathcal{G}_{d+1}(\bar{y},j)=\emptyset$, note that  $\mathcal{G}_{d+2}(\bar{y},j)$ and  $\mathcal{G}_{d+1}(y,\bar{j})$ are also cross-intersecting. Thus,
$
|\mathcal{G}_{d+1}(y,\bar{j})|\leq\binom{n-2}{d}-\binom{n-d-3}{d}.
$
This, together with (\ref{c8}) and (\ref{c9}), implies 
\begin{align*}
|\mathcal{F}|=&\sum_{i=0}^{d-1}|\mathcal{G}_{i}|+|\mathcal{G}_d(j)|+|\mathcal{G}_d(\bar{j})|+|\mathcal{G}_{d+1}(y,j)|+|\mathcal{G}_{d+1}(y,\bar{j})|+1\\
\leq&\sum_{i=0}^d\binom{n}{i}+\binom{n-1}{d}-\binom{n-d-2}{d}-\binom{n-d-3}{d}+1.
\end{align*}

{\bf Case 2. Assume that $\mathcal{G}_{d+2}(y, j)\neq \emptyset$.}
It then follows that  $|\mathcal{G}_{d+2}|=|\mathcal{G}_{d+2}(y, j)|+1$ and $\mathcal{G}_{d+2}(\bar{j})=\emptyset$.
Since $\mathcal{G}_{d+2}(\bar{y},j)$ and  $\mathcal{G}_{d}(y, \bar{j})$ are cross-intersecting, we have
$
|\mathcal{G}_{d}(y, \bar{j})|\leq\binom{n-2}{d-1}-\binom{n-d-3}{d-1}.
$
Note that
$
|\mathcal{G}_{d+1}(\bar{y},\bar{j})|=0, |\mathcal{G}_{d+1}(y,j)|\leq \binom{n-2}{d-1},  |\mathcal{G}_{d}(j)|\leq \binom{n-1}{d-1}
$.
We now split the proof into two subcases.

{\bf Subcase 2.1. Assume that $\mathcal{G}_{d}(\bar{y},\bar{j})\neq \emptyset$.}
Since $\mathcal{G}_{d+2}(y, j)\neq \emptyset$ and $n\geq 2d+3$,  applying Lemma \ref{H672} yields
\begin{align*}
|\mathcal{G}_{d+2}(y,j)|+|\mathcal{G}_{d}(\bar{y},\bar{j})|\leq\binom{n-2}{d}-\binom{n-d-2}{d}+1.
\end{align*}

If  $\mathcal{G}_{d+1}(y,\bar{j})$ and  $\mathcal{G}_{d+1}(\bar{y},j)$ are both non-empty, then applying Lemma \ref{H672} yields
$$
|\mathcal{G}_{d+1}(y,\bar{j})|+|\mathcal{G}_{d+1}(\bar{y},j)|\leq\binom{n-2}{d}-\binom{n-d-2}{d}+1.
$$
It follows that
\begin{align*}
|\mathcal{F}|=& \sum_{i=0}^{d-1}|\mathcal{G}_{i}|+|\mathcal{G}_d(j)|+|\mathcal{G}_d(y,\bar{j})|+|\mathcal{G}_{d+2}(y,j)|+|\mathcal{G}_{d}(\bar{y},\bar{j})|\\
&+|\mathcal{G}_{d+1}(y,j)|+|\mathcal{G}_{d+1}(y,\bar{j})|+|\mathcal{G}_{d+1}(\bar{y},j)|+1\\
\leq&\sum_{i=0}^d\binom{n}{i}+\binom{n-1}{d}-2\binom{n-d-2}{d}-\binom{n-d-3}{d-1}+3,
\end{align*}
which does not exceed the required upper bound.

If  $\mathcal{G}_{d+1}(y,\bar{j})=\emptyset$, then $\mathcal{G}_{d+1}(\bar{j})=\emptyset$. Since
$
|\mathcal{G}_{d}(\bar{y},\bar{j})|\leq\binom{n-2}{d}-\binom{n-d-2}{d},
$
 we have
\begin{align*}
|\mathcal{G}_{d}(\bar{j})|=|\mathcal{G}_{d}(y, \bar{j})|+|\mathcal{G}_{d}(\bar{y},\bar{j})|\leq \binom{n-1}{d}-\binom{n-d-2}{d}-\binom{n-d-3}{d-1}.
\end{align*}
Since $\mathcal{G}_{d+2}(\bar{j})=\emptyset$, we further have
\begin{align*}
|\mathcal{G}(\bar{j})|=\sum_{i=0}^{d-1}|\mathcal{G}_i(\bar{j})|+|\mathcal{G}_{d}(\bar{j})| \leq\sum_{i=0}^{d-1}\binom{n-1}{i}+\binom{n-1}{d}-\binom{n-d-2}{d}-\binom{n-d-3}{d-1}.
\end{align*}
Thus, the result follows from Proposition \ref{prop1}.

If  $\mathcal{G}_{d+1}(\bar{y},j)=\emptyset$, then $\mathcal{G}_{d+1}(\bar{y})=\emptyset$. Since
$
|\mathcal{G}_{d}(\bar{y},\bar{j})|\leq\binom{n-2}{d}-\binom{n-d-2}{d}
$
and $|\mathcal{G}_{d}(\bar{y},j)|\leq \binom{n-2}{d-1}$, we have
\begin{align*}
|\mathcal{G}_{d}(\bar{y})|\leq\binom{n-1}{d}-\binom{n-d-2}{d}.
\end{align*}
Combining this with $|\mathcal{G}_{d+2}(\bar{y})|=1$,  we obtain
$$
|\mathcal{G}(\bar{y})|=\sum_{i=0}^{d-1}|\mathcal{G}_i(\bar{y})|+|\mathcal{G}_{d}(\bar{y})| +1\leq \sum_{i=0}^{d-1}\binom{n-1}{i}+\binom{n-1}{d}-\binom{n-d-2}{d}+1.
$$
It follows that
\begin{align*}
|\mathcal{F}|\leq& 2|\mathcal{G}(\bar{y})|
\leq \sum_{i=0}^d\binom{n}{i}+\binom{n-1}{d}-2\binom{n-d-2}{d}+2,
\end{align*}
which does not exceed the required upper bound.

{\bf Subcase 2.2. Assume that $\mathcal{G}_{d}(\bar{y},\bar{j})= \emptyset$.}  Recall that $|\mathcal{G}_{d}(y, \bar{j})|\leq\binom{n-2}{d-1}-\binom{n-d-3}{d-1}
$.
By a similar argument as in Lemma \ref{c3.2}, we obtain
\begin{align*}
|\mathcal{F}|\leq {\rm max} &\left\{\sum_{i=0}^d\binom{n}{i}+\binom{n-1}{d}-2\binom{n-d-2}{d}-2\binom{n-d-3}{d-1}, ~\sum_{i=0}^{d-1}\binom{n}{i}+\binom{n-1}{d-1}+\right.\\
&\left.2 \binom{n-2}{d-1}-2\binom{n-d-3}{d-1}, ~\sum_{i=0}^{d-1}\binom{n}{i}+\binom{n-1}{d-1}+2 \binom{n-2}{d-1}+2\right\}\\
=&\begin{cases}\sum_{i=0}^d\binom{n}{i}+\binom{n-1}{d}-2\binom{n-d-2}{d}-2\binom{n-d-3}{d-1} & d\geq 2, \\ 6 & d=1.\end{cases}
\end{align*}
This is smaller than the required upper bound. Thus, we complete the proof of Lemma \ref{c3.6}.
\end{proof}

\begin{remark}\label{remad1}
Let $j\neq y\in [n]$. For $d\geq 1$,  let $D\in \binom{[n]}{d+2}$ with $j\in D, y\notin D$. Consider  the following family:
\begin{align*}
\mathcal{Q}(n, 2 d+1)=&\left\{F\subseteq [n]: |F|\leq d-1\right\}\cup\{D\}\cup\left\{F\in \binom{[n]}{d}: j\in F\right\}\cup\left\{F\in \binom{[n]}{d}: j\notin F, F \cap D\neq \emptyset\right\}\\
&\cup\left\{F\in \binom{[n]}{d+1}: y, j\in F\right\}
\cup\left\{F\in \binom{[n]}{d+1}: y\in F, j\notin F, F \cap D\neq \emptyset\right\}.
\end{align*}
In Lemma \ref{c3.6}, $|\mathcal{F}|$ can be maximized by a translation of $\mathcal{Q}(n, 2 d+1)$.
\end{remark}

\begin{lemma}\label{c3.7}
Suppose that $|\mathcal{G}_{d+1}(\bar{y},\bar{j})|=0$ and $|\mathcal{G}_{d+2}(\bar{y},j)|=0$. Then
\begin{align*}
|\mathcal{F}|\leq \sum_{i=0}^d\binom{n}{i}+\binom{n-1}{d}-2\binom{n-d-2}{d}+2.
\end{align*}
In particular, if $d=1$, then $|\mathcal{F}|\leq 8$.
\end{lemma}
\begin{proof}
If $S_j(\mathcal{G})\subseteq\mathcal{K}(n,2d+1)$,
the result follows from Lemmas~\ref{c3.1} and~\ref{c3.2}.
We may assume that
$S_j(\mathcal{G})\nsubseteq\mathcal{K}(n,2d+1)$.
Since $|\mathcal{G}_{d+2}(\bar{y},j)|=0$,
we have $j,y\in G$ for every $G\in\mathcal{G}_{d+2}$.
Hence,
$|\mathcal{G}_{d+2}(y,j)|
=|\mathcal{G}_{d+2}|\neq0$.
Moreover, we have
$|\mathcal{G}_{d+1}(\bar{y},j)|\geq1$ and
$|\mathcal{G}_{d+1}(\bar{y},j)
\cap\mathcal{G}_{d}(\bar{y},\bar{j})|=1$.
This implies that
$\mathcal{G}_{d}(\bar{y},\bar{j})\neq\emptyset$.
Then the same counting argument as in the proof of
Lemma~\ref{c3.1} gives the required bound.
\end{proof}

\underline{Secondly, let us consider the case $|\mathcal{G}_{d+1}(\bar{y},\bar{j})|=1$.} 
Then  $|\mathcal{G}_{d+2}(y, j)|=|\mathcal{G}_{d+2}|\neq 0$ and $\mathcal{G}_{d+1}(\bar{y},j)\cap \mathcal{G}_{d}(\bar{y},\bar{j})= \emptyset$. 
In particular, $j, y\in G$ for every $G\in \mathcal{G}_{d+2}$. 
{\bf We next finish this case by proving the following two lemmas (Lemmas \ref{c3.8} and \ref{c3.9}).}

\begin{lemma}\label{c3.8}
Suppose that $|\mathcal{G}_{d+1}(\bar{y},\bar{j})|=1$ and $\mathcal{G}_{d}(\bar{y},\bar{j})\neq \emptyset$. Then
\begin{align*}
|\mathcal{F}|\leq \sum_{i=0}^d\binom{n}{i}+\binom{n-1}{d}-2\binom{n-d-2}{d}+2.
\end{align*}
In particular, if $d=1$, then $|\mathcal{F}|\leq 8$.
\end{lemma}
\begin{proof}
Since $\mathcal{G}_{d+2}(y, j)\neq \emptyset$,  we infer that (\ref{c1}) holds.
Clearly,
$
|\mathcal{G}_{d+1}(\bar{y},\bar{j})|=1, ~ |\mathcal{G}_{d}(j)|\leq \binom{n-1}{d-1}, ~ |\mathcal{G}_{d}(y,\bar{j})|\leq \binom{n-2}{d-1}.
$
Moreover, $\mathcal{G}_{d+1}(y,j)$ and  $\mathcal{G}_{d+1}(\bar{y},\bar{j})$ are also cross-intersecting because $\Delta(\mathcal{G})\leq 2d+1$. It follows that
$
|\mathcal{G}_{d+1}(y,j)|\leq \binom{n-2}{d-1}- \binom{n-d-3}{d-1}.
$

If  $\mathcal{G}_{d+1}(y,\bar{j})$ and  $\mathcal{G}_{d+1}(\bar{y},j)$ are both non-empty, then (\ref{c4}) holds.
Similar to (\ref{ct4.1}), we have
\begin{align*}
|\mathcal{F}|\leq& \sum_{i=0}^{d-1}\binom{n}{i}+|\mathcal{G}_d(j)|+|\mathcal{G}_d(y,\bar{j})|+|\mathcal{G}_{d+2}(y,j)|+|\mathcal{G}_{d}(\bar{y},\bar{j})|\\
&+|\mathcal{G}_{d+1}(y,j)|+|\mathcal{G}_{d+1}(y,\bar{j})|+|\mathcal{G}_{d+1}(\bar{y},j)|+1\\
\leq&\sum_{i=0}^d\binom{n}{i}+\binom{n-1}{d}-2\binom{n-d-2}{d}-\binom{n-d-3}{d-1}+3,
\end{align*}
which does not exceed the required upper bound.

   If  $\mathcal{G}_{d+1}(y,\bar{j})=\emptyset$, then $|\mathcal{G}_{d+1}(\bar{j})|=1$. Moreover, (\ref{ca1}) holds and $\mathcal{G}_{d+2}(\bar{j})=\emptyset$. 
Consequently, 
\begin{align*}
|\mathcal{F}|\leq 2\sum_{i=0}^{d-1}|\mathcal{G}_i(\bar{j})|+2|\mathcal{G}_{d}(\bar{j})|+2 
\leq\sum_{i=0}^d\binom{n}{i}+\binom{n-1}{d}-2\binom{n-d-2}{d}+2.
\end{align*}

   If  $\mathcal{G}_{d+1}(\bar{y},j)=\emptyset$, then $|\mathcal{G}_{d+1}(\bar{y})|=1$. Note that (\ref{ca2}) holds and $\mathcal{G}_{d+2}(\bar{y})=\emptyset$. Consequently,
\begin{align*}
|\mathcal{F}|\leq 2\sum_{i=0}^{d-1}|\mathcal{G}_i(\bar{y})|+2|\mathcal{G}_{d}(\bar{y})|+2 
\leq\sum_{i=0}^d\binom{n}{i}+\binom{n-1}{d}-2\binom{n-d-2}{d}+2.
\end{align*}
This completes the proof of Lemma \ref{c3.8}.
\end{proof}

\begin{lemma}\label{c3.9}
Suppose that $|\mathcal{G}_{d+1}(\bar{y},\bar{j})|=1$ and $\mathcal{G}_{d}(\bar{y},\bar{j})= \emptyset$. Then
\begin{align*}
|\mathcal{F}|\leq \sum_{i=0}^d\binom{n}{i}+\binom{n-1}{d}-2\binom{n-d-2}{d}+2.
\end{align*}
In particular, if $d=1$, then $|\mathcal{F}|\leq 8$.
\end{lemma}
\begin{proof}
Put
$\mathcal{E}=\mathcal{G}_{d+1}(\bar{y},j)$ and
$\mathcal{T}=\mathcal{G}_{d+1}(y,\bar{j})$.
These are cross-intersecting families of $d$-sets
on $[n]\setminus\{y,j\}$.

Suppose first that $\mathcal{E}\neq\emptyset$.
If $\mathcal{T}\neq\emptyset$, Lemma~\ref{H672} gives
\[
 |\mathcal{T}|+|\mathcal{E}|
 \leq \binom{n-2}{d}-\binom{n-d-2}{d}+1.
\]
Since $|\mathcal{E}|\geq1$, we obtain
\[
 |\mathcal{T}|
 \leq \binom{n-2}{d}-\binom{n-d-2}{d}.
\]
This bound also holds when $\mathcal{T}=\emptyset$.
By assumption,
$\mathcal{G}_{d}(\bar{y},\bar{j})=\emptyset$ and
$|\mathcal{G}_{d+1}(\bar{y},\bar{j})|=1$.
Also, $\mathcal{G}_{d+2}(\bar{j})=\emptyset$.
Consequently,
\[
 |\mathcal{G}(\bar{j})|
 \leq
 \sum_{i=0}^{d-1}\binom{n-1}{i}
 +\binom{n-2}{d-1}+|\mathcal{T}|+1\leq
 \sum_{i=0}^{d}\binom{n-1}{i}
 -\binom{n-d-2}{d}+1.
\]
Using $|\mathcal{G}(j)|\leq|\mathcal{G}(\bar{j})|$,
we conclude that
 \begin{align*}
 |\mathcal{F}|
 &\leq2|\mathcal{G}(\bar{j})|\leq
 2\sum_{i=0}^{d}\binom{n-1}{i}
 -2\binom{n-d-2}{d}+2\\
 &=\sum_{i=0}^{d}\binom{n}{i}
 +\binom{n-1}{d}-2\binom{n-d-2}{d}+2.
 \end{align*}

Now suppose that $\mathcal{E}=\emptyset$.
Then $|\mathcal{G}_{d+1}(\bar{y})|=1$ and
$\mathcal{G}_{d+2}(\bar{y})=\emptyset$.
Since $\mathcal{G}_{d}(\bar{y},\bar{j})=\emptyset$,
we have
\[
 |\mathcal{G}(\bar{y})|
 \leq
 \sum_{i=0}^{d-1}\binom{n-1}{i}
 +\binom{n-2}{d-1}+1.
\]
Therefore,
 \begin{align*}
 |\mathcal{F}|&\leq2|\mathcal{G}(\bar{y})|\leq
 2\sum_{i=0}^{d-1}\binom{n-1}{i}
 +2\binom{n-2}{d-1}+2\\
 &\leq
 \sum_{i=0}^{d}\binom{n}{i}
 +\binom{n-1}{d}-2\binom{n-d-2}{d}+2,
 \end{align*}
where the last inequality follows from Pascal's identity
and
$\binom{n-2}{d}\geq\binom{n-d-2}{d}$.
\end{proof}

\subsubsection{Down-shift lands in $\mathcal{R}(n,2d)$}

 Suppose that $\mathcal{F}$ is not a complex,  $s=2d$ and there is a family $\mathcal{G}$ obtained from $\mathcal{F}$ by repeated down-shifts  satisfying $\mathcal{G}\nsubseteq\mathcal{K}(n, 2d),\mathcal{H}(n, 2d), \mathcal{R}(n, 2d)$,  in addition, $\mathcal{G}\nsubseteq\mathcal{H}^*(n, 4), \mathcal{R}^*(n, 4), \mathcal{U}^*(n, 4)$ if $s=4$, but $S_j(\mathcal{G})\subseteq\mathcal{R}(n, 2d)$ for some $j \in [n]$.
Let $d\geq 2$ be an integer. There exists $y\in R\in\binom{[n]}{d+2}$ such that
\begin{align*}
S_j(\mathcal{G})\subseteq\mathcal{R}(n, 2 d)=&\left\{F\subseteq [n]: |F|\leq d-2\right\}\cup\{R\} \cup\left\{F\subseteq [n] : y\in F, |F|=d-1 \text{ or } d\right\}\\
&\cup\left\{F\subseteq [n] : y\notin F, F \cap R \neq \emptyset, |F|=d-1 \text{ or } d\right\}.
\end{align*}
By Subsubsection \ref{ssec31}, we may assume that $S_j(\mathcal{G})\nsubseteq\mathcal{K}(n, 2d)$. Then ${\rm max}\{|G|: G\in \mathcal{G}\}\geq d+2$.
Recall that
$
\Delta(\mathcal{G})  \leq 2d
$
and $|\mathcal{G}(i)| \leq |\mathcal{G}(\bar{i})|$ for all $i \in[n]$.

\paragraph{${\rm max}\{|G|: G\in \mathcal{G}\}=d+2$}

Since $S_j(\mathcal{G})\subseteq\mathcal{R}(n, 2d)$ and  $S_j(\mathcal{G})\nsubseteq\mathcal{K}(n, 2d)$, we infer that $\mathcal{G}_{d+2}=\{R\}$ and $S_j(R)=R$. Hence, $j\notin R$ or $j\in R$ and $R\backslash \{j\}\in \mathcal{G}_{d+1}$. Let us note that $R\backslash \{j\}\in \mathcal{G}_{d+1}$ would imply $S_j(R\backslash \{j\})=R\backslash \{j\}$. Then $S_j(\mathcal{G})_{d+1}\neq \emptyset$, which contradicts  $S_j(\mathcal{G})\subseteq\mathcal{R}(n, 2d)$.
Therefore, we have $j\notin R$.

Observe that $\Delta(\mathcal{G})\leq 2d$ implies that $\mathcal{G}_{d+1}(j)$ and $\mathcal{G}_{d}(\bar{j})$ are cross-intersecting. Moreover, $\mathcal{G}_{d+2}$ and $\mathcal{G}_{d}(\bar{j})$ are cross-intersecting, $\mathcal{G}_{d+2}$ and $\mathcal{G}_{d-1}$ are cross-intersecting. Then
$
|\mathcal{G}_{d-1}|\leq\binom{n}{d-1}-\binom{n-d-2}{d-1}.
$

\begin{lemma}\label{c3.10}
Suppose that ${\rm max}\{|G|: G\in \mathcal{G}\}=d+2$ and $j\notin R$. Then
\begin{align*}
|\mathcal{F}|\leq \sum_{i=0}^d\binom{n}{i}-\binom{n-d-1}{d}-\binom{n-d-2}{d-1}+1.
\end{align*}
\end{lemma}
\begin{proof}
Since $j\notin R$, we have $j\neq y$ and $|\mathcal{G}_{d+2}|=|\mathcal{G}_{d+2}(y,\bar{j})|=1$. Observe that $S_j(\mathcal{G})\subseteq\mathcal{R}(n, 2d)$ implies that $j\in G$ for any $G\in \mathcal{G}_{d+1}$ and $\mathcal{G}_{d+1}(j)\cap \mathcal{G}_{d}(\bar{j}) =\emptyset$. Thus, $|\mathcal{G}_{d+1}(\bar{j})|=0$. Since $\mathcal{G}\nsubseteq\mathcal{R}(n, 2d)$, we have $|\mathcal{G}_{d+1}(j)|\neq 0$.
If $|\mathcal{G}_{d}(\bar{j})|\leq 1$, then
$$
|\mathcal{G}(\bar{j})|=\sum_{i=0}^d|\mathcal{G}_i(\bar{j})|+|\mathcal{G}_{d+2}(\bar{j})|\leq \sum_{i=0}^{d-1}\binom{n-1}{i}+2.
$$
By  Proposition \ref{prop1}, we have
$$
|\mathcal{F}|\leq \sum_{i=0}^{d-1}\binom{n}{i}+\binom{n-1}{d-1}+4.
$$
Since $d\geq 2$ and $n\geq 2d+2$,  this  bound is smaller than the one we want.
The rest of the proof follows that of Lemma~\ref{c3.5}.
\end{proof}
By Lemma \ref{c3.10}, we complete the proof of the case ${\rm max}\{|G|: G\in \mathcal{G}\}=d+2$.

\paragraph{${\rm max}\{|G|: G\in \mathcal{G}\}\geq d+3$}

Since $S_j(\mathcal{G})\subseteq\mathcal{R}(n, 2d)$, we infer that ${\rm max}\{|G|: G\in \mathcal{G}\}= d+3$, $\mathcal{G}_{d+3}=\{R^{\prime}\}$ for some $j\in R^{\prime}$ and $S_j(R^{\prime})=R^{\prime}\backslash \{j\}$. Moreover, $y\in R^{\prime}\backslash \{j\}$, $\mathcal{G}_{d+2}=\emptyset$, $j\in G$ for any $G\in \mathcal{G}_{d+1}$ and $\mathcal{G}_{d+1}(j)\cap \mathcal{G}_{d}(\bar{j}) =\emptyset$.
Observe that $\Delta(\mathcal{G})\leq 2d$ implies that $\mathcal{G}_{d+1}(j)$ and $\mathcal{G}_{d}(\bar{j})$ are cross-intersecting. Furthermore, $\mathcal{G}_{d+3}$ and $\mathcal{G}_{d}$ are cross-$2$-intersecting, $\mathcal{G}_{d+3}(j)$ and $\mathcal{G}_{d}(j)$ are cross-intersecting, $\mathcal{G}_{d+3}$ and $\mathcal{G}_{d-1}$ are cross-intersecting, $\mathcal{G}_{d+3}$ and $\mathcal{G}_{d-2}$ are cross-intersecting. It follows that
\begin{align}\label{c13}
\begin{split}
&|\mathcal{G}_d|\leq\binom{n}{d}-\binom{n-d-3}{d}-(d+3)\binom{n-d-3}{d-1},~|\mathcal{G}_d(j)|\leq\binom{n-1}{d-1}-\binom{n-d-3}{d-1},\\
&|\mathcal{G}_{d-1}|\leq\binom{n}{d-1}-\binom{n-d-3}{d-1},~|\mathcal{G}_{d-2}|\leq\binom{n}{d-2}-\binom{n-d-3}{d-2}.
\end{split}
\end{align}

\begin{lemma}\label{c3.11}
Suppose that ${\rm max}\{|G|: G\in \mathcal{G}\}= d+3$. Then
\begin{align*}
|\mathcal{F}|<\sum_{i=0}^d\binom{n}{i}-\binom{n-d-1}{d}-\binom{n-d-2}{d-1}+2.
\end{align*}
\end{lemma}
\begin{proof}
If $|\mathcal{G}_{d}(\bar{j})|\leq 1$, note that $|\mathcal{G}_{d+1}(\bar{j})|=|\mathcal{G}_{d+2}(\bar{j})|=|\mathcal{G}_{d+3}(\bar{j})|=0$. Then
$$
|\mathcal{G}(\bar{j})|=\sum_{i=0}^d|\mathcal{G}_i(\bar{j})|\leq \sum_{i=0}^{d-1}\binom{n-1}{i}+1.
$$
In view of Proposition \ref{prop1}, the result holds.

If $|\mathcal{G}_{d+1}(j)|\leq 1$, note that $|\mathcal{G}_{d+1}|=|\mathcal{G}_{d+1}(j)|, |\mathcal{G}_{d+2}|=0, |\mathcal{G}_{d+3}|=1$. Applying (\ref{c13})  yields
\begin{align*}
|\mathcal{F}|\leq &  \sum_{i=0}^{d-3}\binom{n}{i}+|\mathcal{G}_{d-2}|+|\mathcal{G}_{d-1}|+|\mathcal{G}_{d}|+2\\
\leq&\sum_{i=0}^d\binom{n}{i}-\binom{n-d-1}{d}-(d+2)\binom{n-d-3}{d-1}+2.
\end{align*}
Since $n\geq 2d+2$, by simple calculation, we have $\binom{n-d-3}{d-2}<(d+1)\binom{n-d-3}{d-1}$. It follows that
\begin{align*}
|\mathcal{F}|
<\sum_{i=0}^d\binom{n}{i}-\binom{n-d-1}{d}-\binom{n-d-2}{d-1}+2.
\end{align*}

 If $|\mathcal{G}_{d+1}(j)|\geq 2$ and $|\mathcal{G}_{d}(\bar{j})|\geq 2$, then applying Lemmas \ref{W231} and  \ref{HK32} yields
$$
|\mathcal{G}_d(\bar{j})|+|\mathcal{G}_{d+1}(j)|
\leq\binom{n-1}{d}-\binom{n-d-1}{d}-\binom{n-d-2}{d-1}+1.
$$
Combining this with (\ref{c13}), we obtain
\begin{align*}
|\mathcal{F}|=& \sum_{i=0}^{d-3}|\mathcal{G}_i|+|\mathcal{G}_{d-2}|+|\mathcal{G}_{d-1}|+|\mathcal{G}_d(j)|+|\mathcal{G}_d(\bar{j})|+|\mathcal{G}_{d+1}(j)|+1\\
\leq&\sum_{i=0}^d\binom{n}{i}-\binom{n-d-1}{d}-2\binom{n-d-2}{d-1}-\binom{n-d-3}{d-1}+2,
\end{align*}
which is smaller than the required upper bound.
This completes the proof of Lemma \ref{c3.11}.
\end{proof}
By Lemma \ref{c3.11}, we complete the proof of the case ${\rm max}\{|G|: G\in \mathcal{G}\}\geq d+3$.

\subsection{Part III: Down-shift lands in special structures}\label{se44}
In this part, we consider the cases where after repeated down-shifts, we obtain a family $\mathcal{G}$ such that $S_j(\mathcal{G})$ is contained in one of the special structures $\mathcal{T}(n, 5)$, $\mathcal{H}^*(n,4)$,  $\mathcal{R}^*(n,4)$, $\mathcal{U}^*(n,4)$. Recall that $\mathcal{G}$ satisfies $|\mathcal{G}|=|\mathcal{F}|, \Delta(\mathcal{G}) \leq s$ and 
$|\mathcal{G}(i)| \leq |\mathcal{G}(\bar{i})|$ for all $i \in [n]$.

\subsubsection{Down-shift lands in $\mathcal{T}(n,5)$}

 Suppose that $\mathcal{F}$ is not a complex,  $s=5$ and there is a family $\mathcal{G}$ obtained from $\mathcal{F}$ by repeated down-shifts  satisfying $\mathcal{G}\nsubseteq\mathcal{K}(n, 5),\mathcal{H}(n, 5)$, in addition, $\mathcal{G}\nsubseteq\mathcal{T}(n, 5)$, but $S_j(\mathcal{G})\subseteq\mathcal{T}(n, 5)$.
Recall that
\begin{align*}
\mathcal{T}(n, 5)=\left\{F\subseteq [n]: |F|\leq 2\right\}\cup\left\{F \in\binom{[n]}{3}:|F \cap[3]| \geq 2\right\}.
\end{align*}
Clearly,  ${\rm max}\{|G|: G\in \mathcal{G}\}\geq3$.

\paragraph{${\rm max}\{|G|: G\in \mathcal{G}\}=3$}

Since $\Delta(\mathcal{G}) \leq 5$, $\mathcal{G}_{3}$ is  intersecting. Moreover, $\mathcal{G}_{3}$ is not EKR otherwise $\mathcal{G}\subseteq \mathcal{K}(n, 5)$, and $\mathcal{G}_{3}$ is not HM otherwise $\mathcal{G}\subseteq \mathcal{H}(n, 5)$, and $\mathcal{G}_{3} \nsubseteq \mathcal{T}(n, 3)$ otherwise $\mathcal{G}\subseteq \mathcal{T}(n, 5)$. By Theorem \ref{H17}, we have
$$
|\mathcal{G}_{3}|\leq\binom{n-1}{2}-\binom{n-4}{2}-\binom{n-5}{1}+2.
$$
It follows that
\begin{align*}
|\mathcal{F}|=& \sum_{i=0}^3|\mathcal{G}_i|\leq \sum_{i=0}^2\binom{n}{i}+\binom{n-1}{2}-\binom{n-4}{2}-\binom{n-5}{1}+2.
\end{align*}

\paragraph{ ${\rm max}\{|G|: G\in \mathcal{G}\}\geq 4$}

Since $S_j(\mathcal{G})\subseteq\mathcal{T}(n, 5)$, we have ${\rm max}\{|G|: G\in \mathcal{G}\}= 4$.
It suffices to prove the following two lemmas.

\begin{lemma}\label{l61}
Suppose that ${\rm max}\{|G|: G\in \mathcal{G}\}= 4$. If $j\in \{1,2,3\}$, then
\begin{align*}
|\mathcal{F}|< \sum_{i=0}^2\binom{n}{i}+\binom{n-1}{2}-\binom{n-4}{2}-\binom{n-5}{1}+2.
\end{align*}
\end{lemma}
\begin{proof}
If $j\in \{1,2,3\}$, by symmetry, we may assume  $j=1$. Since $S_1(\mathcal{G})\subseteq\mathcal{T}(n, 5)$, we have $\mathcal{G}_{4}(\bar{1})=\emptyset$ and hence $\mathcal{G}_{4}(1)\neq\emptyset$. Moreover, we have
$
\mathcal{G}_{4}(1)\cup \mathcal{G}_{3}(\bar{1})\subseteq \{\{2,3, i\}: i \in[n]\backslash\{1, 2, 3\}\}, \mathcal{G}_{4}(1)\cap \mathcal{G}_{3}(\bar{1})=\emptyset.
$
Since $\Delta(\mathcal{G}) \leq 5$, the families $\mathcal{G}_{4}(1)$ and $\mathcal{G}_{2}(\bar{1})$ are cross-intersecting, and $\mathcal{G}_{3}(1)$ and $\mathcal{G}_{3}(\bar{1})$ are cross-intersecting.
Note that   $\mathcal{G}_{4}(1)$ and $\mathcal{G}_{3}(\bar{1})$ are both $2$-intersecting. Furthermore,  $|\mathcal{G}_{4}(1)|\geq 1$ and $|\mathcal{G}_{4}(\bar{1})|=0$.
By Lemma \ref{F16}, we have
\begin{align}\label{c26}
|\mathcal{G}_{2}(\bar{1})|+|\mathcal{G}_{4}(1)|\leq\binom{n-1}{2}-\binom{n-4}{2}+1,
\end{align}
and if $|\mathcal{G}_{3}(\bar{1})|\geq 1$, then
$$
|\mathcal{G}_{3}(\bar{1})|+|\mathcal{G}_{3}(1)|\leq\binom{n-1}{2}-\binom{n-4}{2}+1.
$$
Consequently, if $|\mathcal{G}_{3}(\bar{1})|\geq 1$, we obtain
\begin{align*}
|\mathcal{F}|=&|\mathcal{G}_0|+|\mathcal{G}_1|+|\mathcal{G}_{2}(1)|+|\mathcal{G}_2(\bar{1})|+|\mathcal{G}_3(\bar{1})|+|\mathcal{G}_{3}(1)|+|\mathcal{G}_{4}(1)|\\
\leq&\sum_{i=0}^2\binom{n}{i}+\binom{n-1}{2}-2\binom{n-4}{2}+2,
\end{align*}
which is smaller than the required upper bound.

Next, if $|\mathcal{G}_{3}(\bar{1})|=0$, note that (\ref{c26}) implies 
$
|\mathcal{G}_{2}(\bar{1})|\leq\binom{n-1}{2}-\binom{n-4}{2}$.
Then
\begin{align}\label{c28}
|\mathcal{G}(\bar{1})|=\sum_{i=0}^2|\mathcal{G}_i(\bar{1})|\leq n+\binom{n-1}{2}-\binom{n-4}{2}.
\end{align}
By Proposition \ref{prop1}, 
we have
\begin{align*}
|\mathcal{F}|\leq  \sum_{i=0}^{2} \binom{n}{i} +\binom{n-1}{2}-2\binom{n-4}{2}< \sum_{i=0}^2\binom{n}{i}+\binom{n-1}{2}-\binom{n-4}{2}-\binom{n-5}{1}+2,
\end{align*}
as required.
\end{proof}

\begin{lemma}\label{l62}
Suppose that ${\rm max}\{|G|: G\in \mathcal{G}\}= 4$. If $j\notin \{1,2,3\}$, then
\begin{align*}
|\mathcal{F}|\leq \sum_{i=0}^2\binom{n}{i}+\binom{n-1}{2}-\binom{n-4}{2}-\binom{n-5}{1}+2.
\end{align*}
\end{lemma}
\begin{proof}
As $S_j(\mathcal{G})\subseteq\mathcal{T}(n, 5)$, we have $\mathcal{G}_{4}(\bar{j})=\emptyset$ and hence $\mathcal{G}_{4}(j)\neq\emptyset$. Moreover, we have
\begin{align}\label{c29}
\begin{split}
&\mathcal{G}_{4}(j)\cup \mathcal{G}_{3}(\bar{j})\subseteq \left\{F \in\binom{[n]\backslash \{j\}}{3}:|F \cap[3]| \geq 2\right\},\\
& \mathcal{G}_{4}(j)\cap \mathcal{G}_{3}(\bar{j})=\emptyset,~\mathcal{G}_{3}(j)\cap \mathcal{G}_{2}(\bar{j})\subseteq \{\{1,2\},\{1,3\},\{2,3\}\}.
\end{split}
\end{align}
Since $\Delta(\mathcal{G}) \leq 5$, the families $\mathcal{G}_{4}(j)$ and $\mathcal{G}_{2}(\bar{j})$ are cross-intersecting, and $\mathcal{G}_{3}(j)$ and $\mathcal{G}_{3}(\bar{j})$ are cross-intersecting.
It follows from $\mathcal{G}_{4}(j)\neq\emptyset$ that
$
|\mathcal{G}_{2}(\bar{j})|\leq\binom{n-1}{2}-\binom{n-4}{2}.
$
If $\mathcal{G}_{3}(\bar{j})=\emptyset$, then the same argument as  (\ref{c28}) works.

Next we consider the case  $\mathcal{G}_{3}(\bar{j})\neq\emptyset$. 
From (\ref{c29}), we get
\begin{align}\label{c30}
|\mathcal{G}_{3}(\bar{j})|+|\mathcal{G}_{4}(j)|\leq 3n-11.
\end{align}

If $1 \leq |\mathcal{G}_{4}(j)| \leq 2n-7$, then (\ref{c30}) implies $1 \leq |\mathcal{G}_{3}(\bar{j})| \leq 3n-12$. Note that $\binom{n-1}{2}-\binom{n-4}{2}+1<\binom{n-1}{2}-\binom{n-5}{1}+1$. Applying Lemma \ref{Last} to $\mathcal{G}_{4}(j)$ and $\mathcal{G}_{2}(\bar{j})$, as well as to $\mathcal{G}_{3}(j)$ and $\mathcal{G}_{3}(\bar{j})$, we obtain
\begin{align*}
|\mathcal{G}_{2}(\bar{j})|+|\mathcal{G}_{4}(j)|\leq\binom{n-1}{2}-\binom{n-4}{2}+1,\ |\mathcal{G}_{3}(\bar{j})|+|\mathcal{G}_{3}(j)|\leq\binom{n-1}{2}-\binom{n-5}{1}+1.
\end{align*}
Consequently,
$
|\mathcal{G}_{2}(\bar{j})|+|\mathcal{G}_{3}(\bar{j})|+|\mathcal{G}_{3}(j)|+|\mathcal{G}_{4}(j)|\leq 2\binom{n-1}{2}-\binom{n-4}{2}-\binom{n-5}{1}+2.
$

If $2n-6 \leq |\mathcal{G}_{4}(j)| \leq 3n-12$, then  (\ref{c30}) implies $1 \leq |\mathcal{G}_{3}(\bar{j})| \leq n-5< 2n-7$. By Lemma \ref{Last}, we have
\begin{align*}
|\mathcal{G}_{2}(\bar{j})|+|\mathcal{G}_{4}(j)|\leq\binom{n-1}{2}-\binom{n-5}{1}+1,\
|\mathcal{G}_{3}(\bar{j})|+|\mathcal{G}_{3}(j)|\leq\binom{n-1}{2}-\binom{n-4}{2}+1.
\end{align*}
Consequently, we still have
$
|\mathcal{G}_{2}(\bar{j})|+|\mathcal{G}_{3}(\bar{j})|+|\mathcal{G}_{3}(j)|+|\mathcal{G}_{4}(j)|\leq 2\binom{n-1}{2}-\binom{n-4}{2}-\binom{n-5}{1}+2.
$

Therefore, we conclude that
\begin{align*}
|\mathcal{F}|=&|\mathcal{G}_0|+|\mathcal{G}_1|+|\mathcal{G}_{2}(j)|+|\mathcal{G}_2(\bar{j})|+|\mathcal{G}_3(\bar{j})|+|\mathcal{G}_{3}(j)|+|\mathcal{G}_{4}(j)|\\
\leq&  \sum_{i=0}^2\binom{n}{i}+\binom{n-1}{2}-\binom{n-4}{2}-\binom{n-5}{1}+2,
\end{align*}
where we used the fact that $|\mathcal{G}_{2}(j)|\leq n-1$. 
\end{proof}

\subsubsection{Down-shift lands in $\mathcal{H}^*(n,4)$}

 Suppose that $\mathcal{F}$ is not a complex,  $s=4$ and there is a family $\mathcal{G}$ obtained from $\mathcal{F}$ by repeated down-shifts  satisfying $\mathcal{G}\nsubseteq\mathcal{K}(n, 4), \mathcal{H}(n, 4), \mathcal{R}(n, 4), \mathcal{R}^*(n, 4), \mathcal{U}^*(n, 4)$, in addition, $\mathcal{G}\nsubseteq\mathcal{H}^*(n, 4)$, but $S_j(\mathcal{G})\subseteq\mathcal{H}^*(n, 4)$.
Recall that
\begin{align*}
\mathcal{H}^*(n, 4)=\left\{F\subseteq [n]: |F|\leq 1\right\} \cup\left\{F \in\binom{[n]}{2}: F \cap[2] \neq \emptyset\right\} \cup\{\{1,2, i\}: i \in[3, n]\}.
\end{align*}
Moreover, we have $|\mathcal{G}|=|\mathcal{F}|, ~\Delta(\mathcal{G}) \leq  4$ and $|\mathcal{G}(i)| \leq |\mathcal{G}(\bar{i})|$ for all $i \in[n]$.
Since $\mathcal{G}\nsubseteq\mathcal{K}(n, 4)$, we have  ${\rm max}\{|G|: G\in \mathcal{G}\}\geq 3$. 

\paragraph{${\rm max}\{|G|: G\in \mathcal{G}\}=3$}

Since $\mathcal{G}\nsubseteq\mathcal{K}(n, 4)$ and $\mathcal{G}\nsubseteq\mathcal{H}(n, 4)$, we have $|\mathcal{G}_{3}|\geq 2$. If $j=2$, then $S_2(\mathcal{G})\subseteq\mathcal{H}^*(n, 4)$ implies that $\mathcal{G}_{3}(\bar{2})=\emptyset$, $1\in G$ for any $G\in \mathcal{G}_{3}(2)$ and $1\in G^{\prime}$ for any $G^{\prime}\in \mathcal{G}_{2}(\bar{2})$.
It follows that  $\mathcal{G}\subseteq\mathcal{H}^*(n, 4)$, a contradiction. Hence, we have $j\neq 2$. By symmetry, we also have $j\neq 1$.

Next assume that $j\neq 1, 2$. Then $S_j(\mathcal{G})\subseteq\mathcal{H}^*(n, 4)$ implies that $1,2 \in G$ for any $G\in \mathcal{G}_{3}(\bar{j})$. Moreover, $A\cap [2]\neq \emptyset$ for any $A\in \mathcal{G}_{2}(\bar{j})$. Since $\Delta(\mathcal{G}) \leq 4$, we infer that $\mathcal{G}_{3}(\bar{j})$ and $\mathcal{G}_{2}(j)$ are cross-intersecting.  If $\mathcal{G}_{3}(j)=\emptyset$, then $|\mathcal{G}_{3}(\bar{j})|=|\mathcal{G}_{3}|\geq 2$. Thus,  $B\cap [2]\neq \emptyset$ for any $B\in \mathcal{G}_{2}(j)$.
It follows that $G\cap [2]\neq \emptyset$ for any $G\in \mathcal{G}_{2}$. Recall that $1,2 \in G$ for any $G\in \mathcal{G}_{3}(\bar{j})=\mathcal{G}_{3}$. Therefore, we have $\mathcal{G}\subseteq\mathcal{H}^*(n, 4)$, a contradiction. Consequently, $\mathcal{G}_{3}(j)\neq \emptyset$. Observe that $\Delta(\mathcal{G}) \leq 4$ implies that $\mathcal{G}_{3}(j)$ and $\mathcal{G}_{2}(\bar{j})$ are cross-intersecting.
It suffices to prove the following lemma.

\begin{lemma}\label{c3.12}
Suppose that ${\rm max}\{|G|: G\in \mathcal{G}\}= 3$ and $\mathcal{G}_{3}(j)\neq \emptyset$. Then
\begin{align*}
|\mathcal{F}|\leq \sum_{i=0}^2\binom{n}{i}-\binom{n-3}{2}-\binom{n-4}{1}+2.
\end{align*}
\end{lemma}
\begin{proof}
Recall that $S_j(\mathcal{G})\subseteq\mathcal{H}^*(n, 4)$. If $S_j(G_0)=G_0$ for some $j\in G_0\in \mathcal{G}_{3}$, then $1, 2\in G_0$. Thus, $G_0=\{1,2, j\}$. If $S_j(G_1)=G_1\backslash\{j\}$ for some $j\in G_1\in \mathcal{G}_{3}$, then $G_1\backslash\{j\}\cap [2]\neq \emptyset$. Therefore, $G\cap [2] \neq \emptyset$ for any $G\in \mathcal{G}_{3}(j)$ and $\mathcal{G}_{3}(j)\cap \mathcal{G}_{2}(\bar{j})\subseteq\{\{1,2\} \}$.
We now split the proof into two cases.

{\bf Case 1. Assume that $\mathcal{G}_{3}(\bar{j})=\emptyset$.}
In this case, we have $|\mathcal{G}_{3}(j)|\geq 2$.
If $|\mathcal{G}_{2}(\bar{j})|\leq 1$, then
\begin{align}\label{c16}
|\mathcal{G}(\bar{j})|= \sum_{i=0}^3|\mathcal{G}_i(\bar{j})|\leq \sum_{i=0}^1\binom{n-1}{i}+1,
\end{align}
and the result follows from Proposition \ref{prop1}.
If $|\mathcal{G}_{2}(\bar{j})|\geq 2$, note that $|\mathcal{G}_{3}(j)\cap \mathcal{G}_{2}(\bar{j})|\leq 1$ and $n\geq 6$. By Lemma  \ref{HK32}, we have
$$
|\mathcal{G}_2(\bar{j})|+|\mathcal{G}_{3}(j)|
\leq\binom{n-1}{2}-\binom{n-3}{2}-\binom{n-4}{1}+1.
$$
Consequently,
\begin{align}\label{c15}
\begin{split}
|\mathcal{F}|=& \sum_{i=0}^1|\mathcal{G}_i|+|\mathcal{G}_2(j)|+|\mathcal{G}_2(\bar{j})|+|\mathcal{G}_{3}(j)|\\
\leq&\sum_{i=0}^1\binom{n}{i}+n-1+\binom{n-1}{2}-\binom{n-3}{2}-\binom{n-4}{1}+1,
\end{split}
\end{align}
which is smaller than the required upper bound.

{\bf Case 2. Assume that  $\mathcal{G}_{3}(\bar{j})\neq\emptyset$.}
Recall that $1,2 \in G$ for any $G\in \mathcal{G}_{3}(\bar{j})$. Thus,  $|\mathcal{G}_{3}(\bar{j})|\leq n-3$. Observe  that $\mathcal{G}_{3}(\bar{j})$ and $\mathcal{G}_{2}(j)$ are cross-intersecting, and $\mathcal{G}_{3}(j)$ and $\mathcal{G}_{2}(\bar{j})$ are cross-intersecting and $|\mathcal{G}_{3}(j)\cap \mathcal{G}_{2}(\bar{j})|\leq 1$.

If $|\mathcal{G}_{3}(\bar{j})|=1$, then $|\mathcal{G}_{2}(j)|\leq \binom{n-1}{1}-\binom{n-4}{1}=3$. When $|\mathcal{G}_{2}(\bar{j})|\leq 1$, we have
$$
|\mathcal{G}(\bar{j})|= \sum_{i=0}^3|\mathcal{G}_i(\bar{j})|\leq \sum_{i=0}^1\binom{n-1}{i}+2,
$$
and the result follows from Proposition \ref{prop1}.
When $|\mathcal{G}_{2}(\bar{j})|\geq 2$, note that $\mathcal{G}_{3}(j)\neq \emptyset$. By Lemma \ref{H672}, we have
$$
|\mathcal{G}_2(\bar{j})|+|\mathcal{G}_{3}(j)|
\leq\binom{n-1}{2}-\binom{n-3}{2}+1.
$$
Combining this with $|\mathcal{G}_{3}(\bar{j})|=1$ and $|\mathcal{G}_{2}(j)|\leq 3$, we obtain
\begin{align*}
|\mathcal{F}|=& \sum_{i=0}^1|\mathcal{G}_i|+|\mathcal{G}_2(j)|+|\mathcal{G}_2(\bar{j})|+|\mathcal{G}_{3}(j)|+|\mathcal{G}_{3}(\bar{j})|\\
\leq&\sum_{i=0}^2\binom{n}{i}-\binom{n-3}{2}-\binom{n-4}{1}+2.
\end{align*}

If $|\mathcal{G}_{3}(\bar{j})|\geq 2$, then $B\cap [2]\neq \emptyset$ for any $B\in \mathcal{G}_{2}(j)$ and hence $|\mathcal{G}_{2}(j)|\leq 2$. As $A\cap [2]\neq \emptyset$ for any $A\in \mathcal{G}_{2}(\bar{j})$, we have $G\cap [2]\neq \emptyset$ for any $G\in \mathcal{G}_{2}$.
Since $1,2 \in G$ for any $G\in \mathcal{G}_{3}(\bar{j})$ and $G\cap [2] \neq \emptyset$ for any $G\in \mathcal{G}_{3}(j)$, there exists $x\in [n]\backslash \{1, 2, j\}$  such that $\{1,x\} \in \mathcal{G}_{3}(j)$ or
$\{2,x\} \in \mathcal{G}_{3}(j)$ otherwise  $\mathcal{G}\subseteq\mathcal{H}^*(n, 4)$. Without loss of generality, we assume that
$\{1,x\} \in \mathcal{G}_{3}(j)$.

If $|\mathcal{G}_{2}(\bar{j})|=0$, then $\mathcal{G}\subseteq\mathcal{R}^*(n, 4)$, a contradiction.

If $|\mathcal{G}_{2}(\bar{j})|=1$, then $\mathcal{G}\nsubseteq\mathcal{R}^*(n, 4)$ implies that $\mathcal{G}_{2}(\bar{j})\neq \{\{1,2\}\}$. Recall that $A\cap [2]\neq \emptyset$ for any $A\in \mathcal{G}_{2}(\bar{j})$, $\mathcal{G}_{3}(j)$ and $\mathcal{G}_{2}(\bar{j})$ are cross-intersecting and $\mathcal{G}_{3}(j)\cap \mathcal{G}_{2}(\bar{j})= \{\{1,2\} \}\text{ or } \emptyset$. So $\mathcal{G}_{2}(\bar{j})= \{\{1,u\} \}\text{ or } \{\{2,x\}\}$ for some $u\in [n]\backslash \{1, 2, x,j\}$.
Suppose that $\mathcal{G}_{2}(\bar{j})= \{\{1,u\}\}$. Then
$
\mathcal{G}_{3}(j)\subseteq \{\{1,i\}, i\in [n]\backslash \{1,u, j\} \}\cup  \{\{2, u\}\}.
$
Hence, we have
$
|\mathcal{G}_2(\bar{j})|+|\mathcal{G}_{3}(j)|
\leq n-1.
$
Observe that
$\sum_{i=0}^2\binom{n}{i}-\binom{n-3}{2}-\binom{n-4}{1}+2=3n+1.$
In view of  $|\mathcal{G}_{3}(\bar{j})|\leq n-3$ and $|\mathcal{G}_{2}(j)|\leq 2$, we obtain
\begin{align*}
|\mathcal{F}|=\sum_{i=0}^1|\mathcal{G}_i|+|\mathcal{G}_2(j)|+|\mathcal{G}_2(\bar{j})|+|\mathcal{G}_{3}(j)|+|\mathcal{G}_{3}(\bar{j})|\leq 3n-1<3n+1.
\end{align*}
Suppose that $\mathcal{G}_{2}(\bar{j})= \{\{2,x\}\}$. Then
$
\mathcal{G}_{3}(j)\subseteq \{\{2,i\}, i\in [n]\backslash \{2, x, j\} \}\cup  \{\{1, x\}\}.
$
Exactly the same argument as the case $\mathcal{G}_{2}(\bar{j})=\{ \{1,u\}\}$ works.

Finally, we consider the case $|\mathcal{G}_{2}(\bar{j})|\geq 2$.
If $\mathcal{G}_{3}(j)=\{\{1,x\}\}$, then
$
\mathcal{G}_{2}(\bar{j})\subseteq \{\{1,i\}, i\in [n]\backslash \{1,x, j\} \}\cup  \{\{2, x\}\}.
$
Exactly the same argument as the case $\mathcal{G}_{2}(\bar{j})= \{\{1,u\}\}$ works. So we may assume that $|\mathcal{G}_{3}(j)|\geq 2$.
By Lemma \ref{HK32}, we have
$$
|\mathcal{G}_2(\bar{j})|+|\mathcal{G}_{3}(j)|
\leq\binom{n-1}{2}-\binom{n-3}{2}-\binom{n-4}{1}+1.
$$
Let us note that $|\mathcal{G}_2(j)|+|\mathcal{G}_{3}(\bar{j})|\leq 2+n-3=n-1$.
By replacing  $|\mathcal{G}_2(j)|$ in (\ref{c15}) with $|\mathcal{G}_2(j)|+|\mathcal{G}_{3}(\bar{j})|$ here, we get
\begin{align*}
|\mathcal{F}|= \sum_{i=0}^1|\mathcal{G}_i|+|\mathcal{G}_2(j)|+|\mathcal{G}_3(\bar{j})|+|\mathcal{G}_2(\bar{j})|+|\mathcal{G}_{3}(j)|\leq\sum_{i=0}^2\binom{n}{i}-\binom{n-3}{2}-\binom{n-4}{1}+1,
\end{align*}
which is smaller than the required upper bound.
This completes the proof of Lemma \ref{c3.12}.
\end{proof}

\paragraph{${\rm max}\{|G|: G\in \mathcal{G}\}\geq 4$}

Since $S_j(\mathcal{G})\subseteq\mathcal{H}^*(n, 4)$, we infer that  ${\rm max}\{|G|: G\in \mathcal{G}\}=4$, $j\neq 1, 2$ and $1,2, j\in G$ for any $G\in \mathcal{G}_{4}$. Moreover, 
$1,2 \in G^{\prime}$ for any $G^{\prime} \in \mathcal{G}_{3}(\bar{j})$ and $\mathcal{G}_{4}(j)\cap\mathcal{G}_{3}(\bar{j})=\emptyset$.
It follows that
\begin{align}\label{c17}
|\mathcal{G}_4(j)|+|\mathcal{G}_{3}(\bar{j})|\leq n-3.
\end{align}
Furthermore, $G^{\prime\prime}\cap [2] \neq \emptyset$ for any $G^{\prime\prime}\in \mathcal{G}_{3}(j) \text{ or } \mathcal{G}_{2}(\bar{j})$, and $\mathcal{G}_{3}(j)\cap \mathcal{G}_{2}(\bar{j})= \{\{1,2\} \}\text{ or } \emptyset$. This leads to
\begin{align}\label{c18}
|\mathcal{G}_3(j)|+|\mathcal{G}_{2}(\bar{j})|\leq 2n-4.
\end{align}
Since $\Delta(\mathcal{G}) \leq 4$,  $\mathcal{G}_{1}$ and $\mathcal{G}_{4}$ are cross-intersecting. This implies that $|\mathcal{G}_{1}|\leq 4$ if $|\mathcal{G}_{4}|=1$, and $\mathcal{G}_{1}\subseteq\{\{1\},\{2\},\{j\}\}$ if $|\mathcal{G}_{4}|\geq 2$.
To complete the proof, it suffices to prove the following lemma.

\begin{lemma}\label{c3.13}
Suppose that ${\rm max}\{|G|: G\in \mathcal{G}\}= 4$. Then
\begin{align*}
|\mathcal{F}|\leq \sum_{i=0}^2\binom{n}{i}-\binom{n-3}{2}-\binom{n-4}{1}+2.
\end{align*}
\end{lemma}
\begin{proof}
Since $\sum_{i=0}^2\binom{n}{i}-\binom{n-3}{2}-\binom{n-4}{1}+2=3n+1$, it suffices to prove that $|\mathcal{F}|\leq 3n+1$.

{\bf Case 1. Assume that  $\mathcal{G}_{3}(\bar{j})=\emptyset$.}
If $|\mathcal{G}_{2}(\bar{j})|\leq 1$, note that $\mathcal{G}_{4}(\bar{j})=\emptyset$. Then the same argument as (\ref{c16}) works. Now suppose that $|\mathcal{G}_{2}(\bar{j})|\geq 2$.

Clearly, $|\mathcal{G}_2(j)|\leq n-1$.
If $|\mathcal{G}_{4}|=1$, then applying (\ref{c18}) and $|\mathcal{G}_{1}|\leq 4$ yields
\begin{align*}
|\mathcal{F}|=\sum_{i=0}^1|\mathcal{G}_i|+|\mathcal{G}_{2}(j)|+|\mathcal{G}_2(\bar{j})|+|\mathcal{G}_{3}(j)|+|\mathcal{G}_{4}(j)|\leq 3n+1.
\end{align*}
If $|\mathcal{G}_{4}|\geq 2$, then $\mathcal{G}_{1}\subseteq\{\{1\},\{2\},\{j\}\}$.
Thus, $|\mathcal{G}_1|+|\mathcal{G}_{4}|\leq 3+n-3=n$. In this case, if $|\mathcal{G}_{3}(j)|\geq 2$, then Lemma \ref{HK32} implies 
\begin{align*}
|\mathcal{G}_2(\bar{j})|+|\mathcal{G}_{3}(j)|
\leq\binom{n-1}{2}-\binom{n-3}{2}-\binom{n-4}{1}+1.
\end{align*}
By replacing  $|\mathcal{G}_1|$ in (\ref{c15}) with $|\mathcal{G}_1|+|\mathcal{G}_{4}|$ here, we get the desired inequality.

Next we assume that $|\mathcal{G}_{4}|\geq 2$ and $|\mathcal{G}_{3}(j)|\leq 1$. If $|\mathcal{G}_{3}(j)|= 0$, then $\mathcal{G}\subseteq\mathcal{U}^*(n, 4)$, a contradiction.
If $\mathcal{G}_{3}(j)=\{\{1,2\}\}$, then
$\mathcal{G}\subseteq\mathcal{U}^*(n,4)$,
a contradiction again.
Thus, $\mathcal{G}_{3}(j)=\{\{1,x\}\}$
or $\mathcal{G}_{3}(j)=\{\{2,x\}\}$
for some $x\in[n]\backslash\{1,2,j\}$.
By symmetry, we may assume that
$\mathcal{G}_{3}(j)=\{\{1,x\}\}$.
Then 
$
\mathcal{G}_{2}(\bar{j})\subseteq \{\{1,i\}, i\in [n]\backslash \{1,x, j\} \}\cup  \{\{2, x\}\}.
$
Hence, we have
$
|\mathcal{G}_2(\bar{j})|+|\mathcal{G}_{3}(j)|
\leq n-1.
$
Note that $|\mathcal{G}_{1}|\leq 3$ and $|\mathcal{G}_{4}|\leq n-3$. Consequently,
\begin{align*}
|\mathcal{F}|=\sum_{i=0}^1|\mathcal{G}_i|+|\mathcal{G}_{2}(j)|+|\mathcal{G}_2(\bar{j})|+|\mathcal{G}_{3}(j)|+|\mathcal{G}_{4}(j)|
\leq 3n-1<3n+1.
\end{align*}

{\bf Case 2. Assume that  $\mathcal{G}_{3}(\bar{j})\neq\emptyset$.}
Recall that $\mathcal{G}_{3}(\bar{j})$ and $\mathcal{G}_{2}(j)$ are cross-intersecting, and $1,2 \in G^{\prime}$ for every $G^{\prime} \in \mathcal{G}_{3}(\bar{j})$.
If $|\mathcal{G}_{3}(\bar{j})|=1$, then $|\mathcal{G}_{2}(j)|\leq \binom{n-1}{1}-\binom{n-4}{1}=3$. If $|\mathcal{G}_{3}(\bar{j})|\geq 2$, then $B\cap [2]\neq \emptyset$ for any $B\in \mathcal{G}_{2}(j)$ and hence $|\mathcal{G}_{2}(j)|\leq 2$. 
In both cases, we have $|\mathcal{G}_{2}(j)|\leq 3$. 
Note that $|\mathcal{G}_{1}|\leq 4$.
In view of (\ref{c17}) and (\ref{c18}), we obtain
\begin{align*}
|\mathcal{F}|=\sum_{i=0}^1|\mathcal{G}_i|+|\mathcal{G}_{2}(j)|+|\mathcal{G}_2(\bar{j})|+|\mathcal{G}_{3}(j)|+|\mathcal{G}_3(\bar{j})|+|\mathcal{G}_{4}(j)|\leq 3n+1.
\end{align*}
This completes the proof of Lemma \ref{c3.13}.
\end{proof}

\subsubsection{Down-shift lands in $\mathcal{R}^*(n,4)$}

Suppose that $\mathcal{F}$ is not a complex,  $s=4$ and there is a family $\mathcal{G}$ obtained from $\mathcal{F}$ by repeated down-shifts  satisfying $\mathcal{G}\nsubseteq\mathcal{K}(n, 4), \mathcal{H}(n, 4), \mathcal{R}(n, 4), \mathcal{H}^*(n, 4), \mathcal{U}^*(n, 4)$, in addition, $\mathcal{G}\nsubseteq\mathcal{R}^*(n, 4)$, but $S_j(\mathcal{G})\subseteq\mathcal{R}^*(n, 4)$.
Since $\sum_{i=0}^2\binom{n}{i}-\binom{n-3}{2}-\binom{n-4}{1}+2=3n+1$, it suffices to prove that $|\mathcal{F}|\leq 3n+1$.
Recall that
\begin{align*}
\mathcal{R}^*(n, 4)=&\left\{F\subseteq [n]: |F|\leq 1\right\} \cup\left\{\{1,2\}, \{1,y\}, \{2,y\}\right\}\cup\left\{F \in\binom{[n]}{3}: y\in F, F \cap[2] \neq \emptyset\right\}\\
& \cup\{\{1,2,i\}: i \in[3, n]\}.
\end{align*}
Moreover, we have
$
|\mathcal{G}|=|\mathcal{F}|,~ \Delta(\mathcal{G}) \leq 4$
and $|\mathcal{G}(i)| \leq |\mathcal{G}(\bar{i})|$ for all $i \in[n]$.
Since $\mathcal{G}\nsubseteq\mathcal{K}(n, 4)$, we have  ${\rm max}\{|G|: G\in \mathcal{G}\}\geq 3$. 

\paragraph{${\rm max}\{|G|: G\in \mathcal{G}\}=3$}

Since $\mathcal{G}\nsubseteq\mathcal{K}(n, 4)$ and $\mathcal{G}\nsubseteq\mathcal{H}(n, 4)$, we have $|\mathcal{G}_{3}|\geq 2$. If $j=2$, then $S_2(\mathcal{G})\subseteq\mathcal{R}^*(n, 4)$ implies that $\mathcal{G}_{3}(\bar{2})\subseteq\{\{1,y, i\}, i\in [n]\backslash\{1,2,y\}\}$ and  $\mathcal{G}_{2}(\bar{2})\subseteq\{\{1,y\}\}$. Moreover, we have $\mathcal{G}_{3}(2)\subseteq\{\{1,y\}\}$. Consequently, 
\begin{align}\label{c19}
\begin{split}
|\mathcal{F}|=&\sum_{i=0}^1|\mathcal{G}_i|+|\mathcal{G}_{2}(2)|+|\mathcal{G}_2(\bar{2})|+|\mathcal{G}_{3}(2)|+|\mathcal{G}_3(\bar{2})|\\
\leq&1+n+n-1+2+n-3=3n-1<3n+1.
\end{split}
\end{align}
By symmetry, the result holds if $j=1$ or $j=y$.
Next suppose that $j\neq 1, 2, y$. Then $S_j(\mathcal{G})\subseteq\mathcal{R}^*(n, 4)$ implies that
\begin{align*}
&\mathcal{G}_{3}(\bar{j})\subseteq \left\{F \in\binom{[n]}{3}: y\in F, j\notin F, F \cap[2] \neq \emptyset\right\} \cup\{\{1,2,i\}: i \in[3, n]\backslash\{j\}\},\\
&\mathcal{G}_{2}(\bar{j})\subseteq \left\{\{1,2\}, \{1,y\}, \{2,y\}\right\}, \mathcal{G}_{3}(j)\subseteq \left\{\{1,2\}, \{1,y\}, \{2,y\}\right\},~ \mathcal{G}_{2}(j)\cap \mathcal{G}_{1}(\bar{j})=\emptyset.
\end{align*}
Thus, $|\mathcal{G}_{2}(j)|+|\mathcal{G}_{1}|\leq n$, $|\mathcal{G}_2(\bar{j})|\leq 3$ and $|\mathcal{G}_3(j)|\leq 3$.
Since $\Delta(\mathcal{G}) \leq 4$, the families $\mathcal{G}_{3}(\bar{j})$ and $\mathcal{G}_{2}(j)$ are cross-intersecting. If $\mathcal{G}_{2}(j)=\emptyset$, then $\mathcal{G}\subseteq\mathcal{R}^*(n, 4)$, a contradiction. Hence, $\mathcal{G}_{2}(j)\neq \emptyset$. Then  $|\mathcal{G}_{3}(\bar{j})|\leq 2n-7$.
It follows that
\begin{align*}
|\mathcal{F}|=&|\mathcal{G}_0|+|\mathcal{G}_1|+|\mathcal{G}_{2}(j)|+|\mathcal{G}_2(\bar{j})|+|\mathcal{G}_{3}(j)|+|\mathcal{G}_3(\bar{j})|\\
\leq&1+n+3+3+2n-7=3n<3n+1.
\end{align*}

\paragraph{${\rm max}\{|G|: G\in \mathcal{G}\}\geq 4$}

Since $S_j(\mathcal{G})\subseteq\mathcal{R}^*(n, 4)$, we  have ${\rm max}\{|G|: G\in \mathcal{G}\}=4$.  If $j=2$, then $S_2(\mathcal{G})\subseteq\mathcal{R}^*(n, 4)$ implies that $\mathcal{G}_{4}(\bar{2})=\emptyset$, $\mathcal{G}_{4}(2)\subseteq\{\{1,y, i\}, i\in [n]\backslash\{1,2,y\}\}$, $\mathcal{G}_{3}(\bar{2})\subseteq\{\{1,y, i\}, i\in [n]\backslash\{1,2,y\}\}$ and  $\mathcal{G}_{4}(2)\cap \mathcal{G}_{3}(\bar{2})=\emptyset$. It follows that
\begin{align*}
|\mathcal{G}_4(2)|+|\mathcal{G}_{3}(\bar{2})|\leq n-3.
\end{align*}
Moreover, we have $\mathcal{G}_{2}(\bar{2})\subseteq\{\{1,y\}\}$, $\mathcal{G}_{3}(2)\subseteq\{\{1,y\}\}$. By replacing  $|\mathcal{G}_{3}(\bar{2})|$ in (\ref{c19}) with $|\mathcal{G}_4(2)|+|\mathcal{G}_{3}(\bar{2})|$ here, we get $|\mathcal{F}|\leq 3n$. By symmetry, the result holds if $j=1$ or $j=y$.

Next suppose that $j\neq 1, 2, y$. Then $S_j(\mathcal{G})\subseteq\mathcal{R}^*(n, 4)$ implies that $\mathcal{G}_{4}(\bar{j})=\emptyset$ and
\begin{align}\label{c20}
\begin{split}
&\mathcal{G}_{4}(j), \mathcal{G}_{3}(\bar{j})\subseteq \left\{F \in\binom{[n]}{3}: y\in F, j\notin F, F \cap[2] \neq \emptyset\right\} \cup\{\{1,2,i\}: i \in[3, n]\backslash\{j\}\},\\
&\mathcal{G}_{4}(j)\cap \mathcal{G}_{3}(\bar{j})=\emptyset,~ \mathcal{G}_{2}(\bar{j})\subseteq \left\{\{1,2\}, \{1,y\}, \{2,y\}\right\},~\mathcal{G}_{3}(j)\subseteq \left\{\{1,2\}, \{1,y\}, \{2,y\}\right\},\\
&\mathcal{G}_{2}(j)\cap \mathcal{G}_{1}(\bar{j})=\emptyset.
\end{split}
\end{align}
Therefore, we have
\begin{align}\label{c21}
|\mathcal{G}_{2}(\bar{j})|\leq 3,~|\mathcal{G}_{3}(j)|\leq 3,~|\mathcal{G}_4(j)|+|\mathcal{G}_{3}(\bar{j})|\leq 3n-11,~ |\mathcal{G}_{2}(j)|+|\mathcal{G}_{1}|\leq n.
\end{align}
Since $\Delta(\mathcal{G}) \leq 4$, the families $\mathcal{G}_{1}$ and $\mathcal{G}_{4}$ are cross-intersecting,
$\mathcal{G}_{4}(j)$ and $\mathcal{G}_{3}(\bar{j})$ are cross-$2$-intersecting, $\mathcal{G}_{3}(\bar{j})$ and $\mathcal{G}_{2}(j)$ are cross-intersecting. It follows from $\mathcal{G}_{4}\neq \emptyset$ that $|\mathcal{G}_{1}|\leq 4$.

If $\mathcal{G}_{2}(j)=\emptyset$, then applying (\ref{c21}) yields
\begin{align*}
|\mathcal{F}|=&|\mathcal{G}_0|+|\mathcal{G}_1|+|\mathcal{G}_2(\bar{j})|+|\mathcal{G}_{3}(j)|+|\mathcal{G}_3(\bar{j})|+|\mathcal{G}_4(j)|\\
\leq&1+4+3+3+3n-11=3n<3n+1.
\end{align*}

If $\mathcal{G}_{2}(j)\neq \emptyset$, note that $\mathcal{G}_{3}(\bar{j})$ and $\mathcal{G}_{2}(j)$ are cross-intersecting. By (\ref{c20}) and  symmetry, we may assume that
\begin{align*}
\mathcal{G}_{3}(\bar{j})\subseteq \left\{F \in\binom{[n]}{3}: 1, y\in F, j\notin F \right\} \cup\{\{1,2,i\}: i \in[3, n]\backslash\{j\}\},
\end{align*}
or there exists $x\in [n]\backslash\{1,2,y, j\}$ such that
$
\mathcal{G}_{3}(\bar{j})\subseteq \{\{1,y,x\}, \{2,y,x\}, \{1,2,x\}\}.
$
Note that $|\mathcal{G}_{1}(\bar{j})|, |\mathcal{G}_{2}(\bar{j})|\leq 3$. If $|\mathcal{G}_{3}(\bar{j})|\leq 2$, then
\begin{align*}
|\mathcal{F}|\leq& 2\sum_{i=0}^3|\mathcal{G}_i(\bar{j})|\leq 2(1+3+3+2)=18<3n+1.
\end{align*}
From now on, assume that $|\mathcal{G}_{3}(\bar{j})|\geq 3$.
Consider first the case where
\begin{align*}
\mathcal{G}_{3}(\bar{j})\subseteq \left\{F \in\binom{[n]}{3}: 1, y\in F, j\notin F \right\} \cup\{\{1,2,i\}: i \in[3, n]\backslash\{j\}\}.
\end{align*}
If, in addition, 
$
\mathcal{G}_{4}(j)\subseteq \left\{F \in\binom{[n]}{3}: 1, y\in F, j\notin F \right\} \cup\{\{1,2,i\}: i \in[3, n]\backslash\{j\}\},
$
then  $\mathcal{G}_{4}(j)\cap \mathcal{G}_{3}(\bar{j})=\emptyset$ implies  $|\mathcal{G}_{4}(j)|+|\mathcal{G}_{3}(\bar{j})|\leq 2n-7$. Otherwise, it follows from (\ref{c20}) that  there exists  $a \in[n]\backslash\{1, 2, y,j\}$  such that $\{2,y,a\} \in \mathcal{G}_{4}(j)$.   In this case,  the cross-$2$-intersecting property of $\mathcal{G}_{4}(j)$ and $\mathcal{G}_{3}(\bar{j})$ imposes the restriction
$
\mathcal{G}_{3}(\bar{j})\subseteq \{\{1,2,a\}, \{1,2,y\}, \{1,y,a\}\}.
$
Since $|\mathcal{G}_{3}(\bar{j})|\geq 3$, we further have
\begin{align*}
\mathcal{G}_{3}(\bar{j})= \{\{1,2,a\}, \{1,2,y\}, \{1,y,a\}\}.
\end{align*}
Given that $\mathcal{G}_{4}(j)$ and $\mathcal{G}_{3}(\bar{j})$ are cross-$2$-intersecting and $\mathcal{G}_{4}(j)\cap \mathcal{G}_{3}(\bar{j})=\emptyset$, it follows that $\mathcal{G}_{4}(j)=\{\{2,y,a\}\}$. Consequently, $|\mathcal{G}_{4}(j)|+|\mathcal{G}_{3}(\bar{j})|=4< 2n-7$.
Next, consider the case  where there exists an element  $x\in [n]\backslash\{1,2,y,j\}$ such that
\begin{align*}
\mathcal{G}_{3}(\bar{j})\subseteq \{\{1,y,x\}, \{2,y,x\}, \{1,2,x\}\}.
\end{align*}
Since $|\mathcal{G}_{3}(\bar{j})|\geq 3$, we  have
\begin{align*}
\mathcal{G}_{3}(\bar{j})= \{\{1,y,x\}, \{2,y,x\}, \{1,2,x\}\}.
\end{align*}
Then the cross-$2$-intersecting property of $\mathcal{G}_{4}(j)$ and $\mathcal{G}_{3}(\bar{j})$ and $\mathcal{G}_{4}(j)\cap \mathcal{G}_{3}(\bar{j})=\emptyset$ imply 
$
\mathcal{G}_{4}(j)\subseteq \{\{1,2,y\}\}.
$
Thus, we have $|\mathcal{G}_{4}(j)|+|\mathcal{G}_{3}(\bar{j})|\leq 4< 2n-7$.
Combining the analysis of the above cases with (\ref{c21}), we obtain
\begin{align*}
|\mathcal{F}|=&|\mathcal{G}_0|+|\mathcal{G}_1|+|\mathcal{G}_{2}(j)|+|\mathcal{G}_2(\bar{j})|+|\mathcal{G}_{3}(j)|+|\mathcal{G}_3(\bar{j})|+|\mathcal{G}_{4}(j)|\\
\leq&1+n+3+3+2n-7=3n<3n+1.
\end{align*}

\subsubsection{Down-shift lands in $\mathcal{U}^*(n,4)$}

Suppose that $\mathcal{F}$ is not a complex,  $s=4$ and there is a family $\mathcal{G}$ obtained from $\mathcal{F}$ by repeated down-shifts  satisfying $\mathcal{G}\nsubseteq\mathcal{K}(n, 4),\mathcal{H}(n, 4), \mathcal{R}(n, 4), \mathcal{H}^*(n, 4), \mathcal{R}^*(n, 4)$, in addition, $\mathcal{G}\nsubseteq\mathcal{U}^*(n, 4)$, but $S_j(\mathcal{G})\subseteq\mathcal{U}^*(n, 4)$.
It suffices to prove that $|\mathcal{F}|\leq 3n+1$.
Recall that
\begin{align*}
\mathcal{U}^*(n, 4)=&\{\emptyset\}\cup \{\{1\}, \{2\}, \{y\}\} \cup\left\{\{y,i\},i \in [n]\backslash \{y\}\right\}\cup\left\{F \in\binom{[n]}{2}: y\notin F, F \cap[2] \neq \emptyset\right\} \\
&\cup\{\{1,2,y\}\}\cup\{\{1,2,y, i\}: i \in [n]\backslash \{1,2,y\}\}.
\end{align*}
Moreover, $|\mathcal{G}|=|\mathcal{F}|,~ \Delta(\mathcal{G})  \leq 4$
and $|\mathcal{G}(i)| \leq |\mathcal{G}(\bar{i})|$ for all $i \in[n]$.
By Subsubsections \ref{ssec31} and \ref{ssec33}, we may assume that $S_j(\mathcal{G})\nsubseteq\mathcal{K}(n, 4)$ and $S_j(\mathcal{G})\nsubseteq\mathcal{H}(n, 4)$. It follows from $S_j(\mathcal{G})\subseteq\mathcal{U}^*(n, 4)$ that  ${\rm max}\{|G|: G\in \mathcal{G}\}\geq 4$.

\paragraph{${\rm max}\{|G|: G\in \mathcal{G}\}=4$}

If $j\in \{1,2,y\}$, by symmetry, we may assume that $j=1$.
Then $S_1(\mathcal{G})\subseteq\mathcal{U}^*(n, 4)$ implies that $\mathcal{G}_{4}(\bar{1})=\emptyset$ and $\mathcal{G}_{3}(\bar{1})=\emptyset$. Hence, $\mathcal{G}_{4}(1)\cap \mathcal{G}_{3}(\bar{1})=\emptyset$. Since $\mathcal{G}_{4}\neq\emptyset$, we have $S_1(\mathcal{G}_{4})=\{\{1, 2, y\}\}$, a contradiction.
Therefore, $j\notin \{1,2,y\}$.
In view of $S_j(\mathcal{G})\subseteq\mathcal{U}^*(n, 4)$, we obtain
\begin{align}\label{c22}
\begin{split}
&\mathcal{G}_{4}(\bar{j})\subseteq \{\{1,2,y, i\}: i \in[n]\backslash\{1, 2, y, j\}\},~\mathcal{G}_{4}(j),~ \mathcal{G}_{3}(\bar{j})\subseteq \{\{1,2,y\}\},\\
& \mathcal{G}_{3}(j)\cup \mathcal{G}_{2}(\bar{j})\subseteq \left\{\{y,i\},i \in [n]\backslash \{y,j\}\right\}\cup\left\{F \in\binom{[n]}{2}: y, j\notin F, F \cap[2] \neq \emptyset\right\},\\
&\mathcal{G}_{3}(j)\cap \mathcal{G}_{2}(\bar{j})=\emptyset,~\mathcal{G}_{2}(j),~ \mathcal{G}_{1}(\bar{j})\subseteq \left\{\{1\}, \{2\}, \{y\}\right\}.
\end{split}
\end{align}
Moreover, at most one of $\emptyset$ and $\{j\}$
belongs to $\mathcal{G}$. Hence,
$|\mathcal{G}_0|+|\mathcal{G}_1(j)|\leq1$.

Suppose that $\mathcal{G}_3(j)=\emptyset$.
Then $\mathcal{G}\nsubseteq\mathcal{U}^*(n,4)$
and (\ref{c22}) imply that $\{j\}\in\mathcal{G}$.
Consequently, $\emptyset\notin\mathcal{G}$.
Since $\Delta(\mathcal{G})\leq4$, we also have
$\mathcal{G}_4(\bar{j})=\emptyset$.
Furthermore, (\ref{c22}) gives
\[
 |\mathcal{G}_2(\bar{j})|
 \leq \binom{n-1}{2}-\binom{n-4}{2}
 =3n-9.
\]
Therefore,
\[
 \begin{aligned}
 |\mathcal{F}|
 &=
 |\mathcal{G}_1(j)|+|\mathcal{G}_1(\bar{j})|
 +|\mathcal{G}_2(j)|+|\mathcal{G}_2(\bar{j})|
 +|\mathcal{G}_3(\bar{j})|+|\mathcal{G}_4(j)|\\
 &\leq1+3+3+(3n-9)+1+1\\
 &=3n<3n+1.
 \end{aligned}
\]
Thus, we may assume that
$\mathcal{G}_3(j)\neq\emptyset$.

 If $|\mathcal{G}_{2}(\bar{j})|\leq 2$, then applying (\ref{c22}) yields
\begin{align}\label{c24}
|\mathcal{F}|\leq& 2\sum_{i=0}^4|\mathcal{G}_i(\bar{j})|\leq 2(1+3+2+1+n-4)=2n+6<3n+1.
\end{align}

If  $|\mathcal{G}_{2}(\bar{j})|\geq 3$, note that $\mathcal{G}_{3}(j)$ and $\mathcal{G}_{2}(\bar{j})$ are cross-intersecting. Since  $\mathcal{G}_{3}(j)\neq \emptyset$ and $\mathcal{G}_{3}(j)\cap \mathcal{G}_{2}(\bar{j})=\emptyset$, applying Lemma \ref{H672} to $\mathcal{G}_{3}(j)$ and $\mathcal{G}_{2}(\bar{j})$ gives
$$
|\mathcal{G}_{3}(j)|+| \mathcal{G}_{2}(\bar{j})|\leq \binom{n-1}{2}-\binom{n-3}{2}=2n-5.
$$
This, together with  (\ref{c22}) and $|\mathcal{G}_0|+|\mathcal{G}_1(j)|\leq 1$, implies
\begin{align}\label{c25}
\begin{split}
|\mathcal{F}|=&|\mathcal{G}_0|+|\mathcal{G}_1(j)|+|\mathcal{G}_1(\bar{j})|+|\mathcal{G}_{2}(j)|+|\mathcal{G}_2(\bar{j})|+|\mathcal{G}_{3}(j)|+|\mathcal{G}_3(\bar{j})|+|\mathcal{G}_{4}(j)|+|\mathcal{G}_4(\bar{j})|\\
\leq&1+3+3+2n-5+1+1+n-4=3n<3n+1.
\end{split}
\end{align}

\paragraph{${\rm max}\{|G|: G\in \mathcal{G}\}\geq 5$}

Since $S_j(\mathcal{G})\subseteq\mathcal{U}^*(n, 4)$, we infer that ${\rm max}\{|G|: G\in \mathcal{G}\}= 5$ and  $j\notin \{1,2,y\}$. Additionally, we have $\mathcal{G}_{5}(\bar{j})=\emptyset$, and the following structural constraints hold:
\begin{align}\label{c23}
\begin{split}
&\mathcal{G}_{5}(j)\cup \mathcal{G}_{4}(\bar{j})\subseteq \{\{1,2,y, i\}: i \in[n]\backslash\{1, 2, y, j\}\},~ \mathcal{G}_{5}(j)\cap \mathcal{G}_{4}(\bar{j})=\emptyset,\\
&\mathcal{G}_{4}(j),~ \mathcal{G}_{3}(\bar{j})\subseteq \{\{1,2,y\}\},~\mathcal{G}_{3}(j)\cap \mathcal{G}_{2}(\bar{j})=\emptyset,~\mathcal{G}_{2}(j),~\mathcal{G}_{1}(\bar{j})\subseteq \left\{\{1\}, \{2\}, \{y\}\right\},\\
& \mathcal{G}_{3}(j)\cup \mathcal{G}_{2}(\bar{j})\subseteq \left\{\{y,i\},i \in [n]\backslash \{y,j\}\right\}\cup\left\{F \in\binom{[n]}{2}: y, j\notin F, F \cap[2] \neq \emptyset\right\}.
\end{split}
\end{align}
Moreover, at most one of $\emptyset$ and $\{j\}$
belongs to $\mathcal{G}$. It follows that $|\mathcal{G}_0|+|\mathcal{G}_1(j)|\leq 1$. From $\Delta(\mathcal{G}) \leq 4$, we infer that $\mathcal{G}_{3}(j)$ and $\mathcal{G}_{2}(\bar{j})$ are cross-intersecting,  and $\mathcal{G}_{5}(j)$ and $\mathcal{G}_{2}(\bar{j})$ are cross-$2$-intersecting. Since $\mathcal{G}_{5}(j)\neq \emptyset$, there exists $i_0\in[n]\backslash\{1, 2, y, j\}$ such that $\{1,2,y, i_0\}\in \mathcal{G}_{5}(j)$. By the  cross-$2$-intersecting property of $\mathcal{G}_{5}(j)$ and $\mathcal{G}_{2}(\bar{j})$, we obtain
$$
\mathcal{G}_{2}(\bar{j})\subseteq \left\{\{1,2\}, \{1,i_0\}, \{2,i_0\},\{y,1\}, \{y,2\}, \{y,i_0\}\right\},
$$
which implies $|\mathcal{G}_{2}(\bar{j})|\leq 6$.
If $\mathcal{G}_{3}(j)= \emptyset$, then applying $|\mathcal{G}_0|+|\mathcal{G}_1(j)|\leq 1$, $|\mathcal{G}_{2}(\bar{j})|\leq 6$ and (\ref{c23}) yields
\begin{align*}
|\mathcal{F}|=&|\mathcal{G}_0|+|\mathcal{G}_1(j)|+|\mathcal{G}_1(\bar{j})|+|\mathcal{G}_{2}(j)|+|\mathcal{G}_2(\bar{j})|+|\mathcal{G}_3(\bar{j})|+|\mathcal{G}_{4}(j)|+|\mathcal{G}_4(\bar{j})|+|\mathcal{G}_{5}(j)|\\
\leq&1+3+3+6+1+1+n-4=n+11<3n+1.
\end{align*}
If $\mathcal{G}_{3}(j)\neq \emptyset$, note that $\mathcal{G}_{5}(\bar{j})=\emptyset$ and
$
|\mathcal{G}_4(\bar{j})|+|\mathcal{G}_{5}(j)|\leq n-4$, exactly the same arguments as (\ref{c24}) and (\ref{c25}) work.

\section*{Acknowledgement}
The authors would like to express their sincere thanks to the referee for the valuable suggestions which greatly improved the presentation of the manuscript. 
Yongtao Li was supported by the China Postdoctoral Science
Foundation (No.~2026M793384), the Postdoctoral Fellowship
Program of CPSF (No.~GZC20261656), and the NSFC
(No.~12601675).
Lihua Feng was supported by the NSFC
(Nos.~12271527 and 12471022).
Guihai Yu was supported by the NSFC
(Nos.~12461062 and 11861019) and the Natural Science
Foundation of Guizhou (Nos.~[2020]1Z001 and [2021]5609).

\end{document}